\documentclass[11pt,reqno]{amsart}
\usepackage{amsmath}
\usepackage{amssymb}
\usepackage{todonotes}
\usepackage{setspace}

 \usepackage{amsaddr}
\usepackage{algorithm}
\usepackage{algpseudocode}

\usepackage{overpic}

\usetikzlibrary{decorations.pathreplacing}

\newcommand{\supp}[1]{\mathrm{supp}\left( #1 \right)}

\DeclareMathOperator{\RR}{\mathbb{R}}
\newcommand{\sumx}{\sum_{i,j=0}^{N_x-1}}
\newcommand{\sums}{\sum_{p=0}^{N_s-1}}
\newcommand{\sumphi}{\sum_{q=0}^{N_\phi-1}}

\newcommand{\I}{\mathrm{I}}
\newcommand{\II}{\mathrm{II}}

\newcommand{\mres}{\mathbin{\vrule height 1.6ex depth 0pt width
0.13ex\vrule height 0.13ex depth 0pt width 1.3ex}}

\DeclareMathOperator{\ZZ}{\mathbb{Z}}
\DeclareMathOperator{\NN}{\mathbb{N}}

\DeclareMathOperator{\Radon}{\mathcal{R}}
\DeclareMathOperator{\RayRadon}{\mathcal{R}^\mathrm{rd}_{\delta}}
\DeclareMathOperator{\GenRadon}{\mathcal{R}^\omega_{\delta}}

\DeclareMathOperator{\PixelRadon}{\mathcal{R}^\mathrm{pd}_{\delta}}

\DeclareMathOperator{\RayRadonN}{\mathcal{R}^\mathrm{rd}_{\delta^n}}
\DeclareMathOperator{\PixelRadonN}{\mathcal{R}^\mathrm{pd}_{\delta^n}}

\DeclareMathOperator{\RayWeight}{\omega^\mathrm{rd}_{\delta_\mathit{x}}}
\DeclareMathOperator{\PixelWeight}{\omega^\mathrm{pd}_{\delta_\mathit{s}}}

\DeclareMathOperator{\imgdom}{\Omega}
\DeclareMathOperator{\sinodom}{\mathcal{S}}

\newcommand{\dd}[1]{\,\mathrm{d}{#1}}

\DeclareMathOperator{\RayFq}{F^\mathrm{rd}_q}
\DeclareMathOperator{\PixelFq}{F^\mathrm{pd}_q}

\usepackage[T1]{fontenc}
\usepackage{graphicx}
\newtheorem{theorem}{Theorem}
\numberwithin{theorem}{section} 
\newtheorem{definition}[theorem]{Definition}

\newtheorem{lemma}[theorem]{Lemma}
\newtheorem{remark}[theorem]{Remark}

\newtheorem{corollary}[theorem]{Corollary}

\theoremstyle{remark}

\begin{document}
\title[Convergence of ray- and pixel-driven discretizations]{Convergence of ray- and pixel-driven discretization frameworks in the strong operator topology}

\author[R. Huber]{Richard Huber}
\email{richard.huber@uni-graz.at}
\address{ IDea Lab - The Interdisciplinary Digital Lab at the University of Graz\\
University of Graz, Graz, Austria
}
%

%
\begin{abstract}
Tomography is a central tool in medical applications, allowing doctors to investigate patients' internal features. The Radon transform (in two dimensions) is commonly used to model the measurement process in parallel-beam CT. Suitable discretization of the Radon transform and its adjoint (called the backprojection) is crucial.
The most commonly used discretization approach combines what we refer to as the ray-driven Radon transform with what we refer to as the pixel-driven backprojection, as anecdotal reports describe these as showing the best approximation performance. However, there is little rigorous understanding of induced approximation errors.
These methods involve three discretization parameters: the spatial-, detector-, and angular resolutions. Most commonly, balanced resolutions are used, i.e., the same (or similar) spatial- and detector resolutions are employed.
We present an interpretation of ray- and pixel-driven discretizations as `convolutional methods', a special class of finite-rank operators.
This allows for a structured analysis that can explain observed behavior.
We prove convergence in the strong operator topology of the ray-driven Radon transform and the pixel-driven backprojection under balanced resolutions, thus theoretically justifying this approach. In particular, with high enough resolutions, one can approximate the Radon transform arbitrarily well. Numerical experiments corroborate these theoretical findings.

\thanks{\\ \textbf{Funding:} This work was supported by The Villum Foundation (Grant No.25893) and is in part based on work supported by the International Research Training Group ``Optimization and Numerical Analysis for Partial Differential Equations with Nonsmooth Structures'', funded by the German Research Council (DFG) and Austrian Science Fund (FWF) grant W1244.
}

\end{abstract}
\maketitle
\section{Introduction}

Computed Tomography (CT) is a crucial tool in medicine, allowing the investigation of the interior of patients' bodies \cite{Ammari08_book_methematics_of_medical_imaging,Hsieh_CT_principles}.  
A sequence of X-ray images of the patient from different directions is acquired, from which one reconstructs the three-dimensional distribution of the patient's mass density.
Each measurement point corresponds to the measured loss of intensity (due to attenuation) of an X-ray beam while traversing the body along a straight line. This process can be modeled  via a line integral operator representing the accumulation of attenuation along straight lines.

 In planar parallel beam CT, the measurement process is commonly modeled by the (two-dimensional) Radon transform $\Radon$ \cite{Deans_Radon_applications_1993,helgason1999radon,Natterer:2001:MCT:500773} (in this context, we also refer to it as the forward operator) that maps a function $f$ describing the mass density distribution in the body onto a function $\Radon f$ describing measurements (line integrals) related to all measured straight lines (parametrized by an angular variable $\phi$ and a detector variable $s$). 
 Although modeling different physical processes, the Radon transform (and related operators) also finds application in astrophysics \cite{cameron2004astrotomography}, materials science \cite{leary_analytical_2016}, and seismography \cite{RAWLINSON2010101}.
 
The (parallel-beam) tomographic reconstruction corresponds to the solution of the ill-posed inverse problem $\Radon f =g$ for known measurements $g$ and unknown density distributions $f$.
The filtered backprojection (FBP) \cite{Natterer:2001:MCT:500773}  is an analytical inversion formula that can be used for fast reconstruction. However, as FBP coincides with the application of the (discontinuous) inverse of the Radon transform, it is highly unstable, possessing a high-frequency noise amplification property. (This can, to some degree, be counteracted by low-pass filters \cite{beckmann2015error}.) Thus, more evolved iterative reconstruction techniques were developed, which have an intrinsic regularization effect via early stopping and allow for the easy integration of prior information. These include iterative algebraic reconstruction algorithms (e.g., SIRT and conjugated gradients) \cite{Gilbert_Sirt,SART_ALgo,SCALES_CG_tomography_1987} and variational approaches (e.g., total variation regularized reconstructions) \cite{Scherzer:2008:VMI:1502016,Dong2013,C8NR09058K} that require iterative solution algorithms for convex optimization problems. These iterative methods also involve the adjoint operator $\Radon^*$ (called the backprojection \cite{Natterer:2001:MCT:500773}).

While $\Radon$ and $\Radon^*$ are infinite-dimensional operators, only finite amounts of data can be measured and processed in practical applications. Thus, proper discretization $\Radon_\delta$ (for some discretization parameters $\delta$) is imperative. It is common to think of both measurements and reconstructions as images with pixels of finite resolutions and correspondingly, $\delta=(\delta_x,\delta_\phi,\delta_s)$ denotes the spatial resolution of reconstructions $\delta_x$, and the angular- and detector resolutions $(\delta_\phi,\delta_s)$ of measured data.
 The expectation is that with ever finer resolution $(\delta\to 0)$, also the approximation gets arbitrarily accurate (i.e., $\Radon_\delta \overset{\delta \to 0}{\to} \Radon$ in some sense). Simultaneously, the approximations of the backprojection should also improve with higher degrees of discretization. This is crucial, as it justifies the use of theory concerning the (continuous) Radon transform to discrete settings.
 
Note that the parameters $(\delta_x,\delta_\phi,\delta_s)$ can also be interpreted as sampling parameters and  significant investigations were made into what information can be recovered from which amount and type of sampling; see \cite[Section III]{Natterer:2001:MCT:500773} and \cite{Faridani2004}. In particular, these suggest certain resolution conditions related to the Shannon Sampling theorem (e.g., $\delta_\phi \geq \frac {\delta_s}{2}$) to offer reconstruction guarantees. However, these guarantees require a number of assumptions (e.g., no noise, essentially bandlimited functions) that in practice cannot always be satisfied. Hence, in this work, we will keep the setting general, and not restrict ourselves to certain sampling conditions a priori.

A number of different discretization schemes have been proposed based on different heuristics, showing different strengths and weaknesses.
The most widely used discretization approach employs the ray-driven Radon transform $\RayRadon$  \cite{Siddon1985FastCO,Path_through_pixels,doi:10.1118/1.4761867} and the pixel-driven backprojection $\PixelRadon^*$ \cite{4331812,322963,Qiao2017ThreeNA,doi:10.1137/20M1326635} (we speak of an $\mathrm{rd}$-$\mathrm{pd}^*$ approach). Concerning the choice of discretization parameters $\delta$, it is most common to use similar resolutions for the detector and the reconstruction (we speak of balanced resolutions), i.e., $\delta_x\approx \delta_s$. Note that usually $\delta_s$ is a physical quantity (the width of the physical detector's pixels) and can thus not be influenced. The angular resolution $\delta_\phi$ is chosen during the measurement process. Finally, the spatial resolution $\delta_x$ is fully under our control when doing the reconstructions. The ray-driven approach discretizes line integrals by summation of integrals on line intersections with pixels, while the pixel-driven backprojection is based on linear interpolation on the detector.
There is also a ray-driven backprojection and a pixel-driven Radon transform as the adjoints to the mentioned operators; however, these are said to perform poorly (supposedly creating artifacts \cite{Man_2004,Xie_CUDA_paralelization_2015,Liu_GPU_DDP_2017}) and are thus hardly ever used in practice.
Other discretization schemes include distance-driven methods \cite{1239600,Man_2004} and so-called fast schemes \cite{averbuch2001fast,KINGSTON20062040}.

We can associate the discretizations $\Radon_\delta$ and $\Radon_\delta^*$ with matrices $A \approx \Radon$ and $B\approx \Radon^*$. One would naturally think that $A^T=B$ (one speaks of a matched pair of operators) as they approximate adjoint operators, but this is not the case if the forward and backward discretization from two different frameworks is used (e.g., $\mathrm{rd}$-$\mathrm{pd}^*$ rather than $\mathrm{rd}$-$\mathrm{rd}^*$ or $\mathrm{pd}$-$\mathrm{pd}^*$). Using an unmatched pair can potentially harm iterative solvers' convergence
 \cite{elfving2018unmatched,dong2019fixing,10.1007/s00245-022-09933-5} as convergence guarantees of many iterative solvers (or iterative optimization algorithms, more generally) are based on adjointness. Thus, such methods might converge more slowly or not fully converge when using non-adjoint (unmatched) operator pairs.

However, this danger seems to be outweighed in practice by the supposed better approximation performance of the ray-driven forward $\RayRadon$ and the pixel-driven backprojection $\PixelRadon^*$.
Using mismatched operators is certainly preferable to discretizations that do not adequately represent the measurement process.
There is little rigorous analysis of the discretizations' approximation errors, and anecdotal knowledge of performance is more prevalent; see e.g., \cite{322963,Qiao2017ThreeNA,Liu_GPU_DDP_2017}.
In \cite{doi:10.1137/20M1326635}, the author rigorously discussed approximation errors for pixel-driven methods in the case that the spatial resolution $\delta_x$ is asymptotically smaller than the detector resolution $\delta_s$, finding convergence in the operator norm, thus justifying the $\mathrm{pd}$-$\mathrm{pd}^*$ approach when $\frac{\delta_x}{\delta_s}\to 0$. 
However, in practice, it is much more common to use balanced resolutions ($\delta_x\approx \delta_s$), in which case these results are not applicable.

This paper will justify the use of $\mathrm{rd}$-$\mathrm{pd}^*$ approaches for balanced resolutions by proving convergence in the strong operator topology (i.e., pointwise convergence). This substantiates heuristic notions of approximation performance. In particular, given any function, the resolutions can be chosen fine enough to approximate the Radon transform (or backprojection) arbitrarily well. Some of these results were  already presented in the author's doctoral thesis \cite{huber2022pixel}. Moreover, we show that convergence of the ray-driven backprojection is obtained if $\delta_s \ll \delta_x$. The main theoretical result Theorem \ref{Thm_approximation_ray_driven} was already announced in \cite{10.1007/978-3-031-92366-1_11} without a rigorous proof, which this paper now provides.

This paper is structured as follows: Section \ref{section_Radon} describes the Radon transform and related notation (Section \ref{section_Radon_notation}) and the investigated discretization frameworks (Section \ref{section_Radon_discretization}). We perform a convergence analysis to investigate the approximation properties of these discretizations in Section \ref{section_Radon_proofs}. Section \ref{section_results} formulates the main theoretical result in Theorem \ref{Thm_approximation_ray_driven}, and the corresponding proofs are presented in Section \ref{section_proofs}.
Finally, Section \ref{section_numerics} presents numerical experiments corroborating these theoretical results.

\section{The Discrete Radon Transform}
\label{section_Radon}
Below, we set the notation, give relevant definitions, as well as introduce the considered discretization frameworks.
\subsection{Preliminaries and notation}
\label{section_Radon_notation}

Throughout this paper, we denote the spatial domain by ${\imgdom:=B(0,1)}\subset \RR^2$ (the two-dimensional open unit ball); one can think of it as the area in which the investigated body is located. All investigations in this paper will be planar, i.e., we ignore the natural third space dimension. The set of all conceivable measurements are contained in the sinogram domain we define next.

\begin{definition}[Sinogram domain]
We define the (parallel-beam) sinogram domain $\sinodom:=[0,\pi[\times ]-1,1[$. Moreover, given $(\phi,s)\in \sinodom$, the associated straight line is $L_{\phi,s}:=\{s \vartheta_\phi+t\vartheta_\phi^\perp \in \RR^2\big | \ t\in \RR\}$, where $\vartheta_\phi:=(\cos(\phi),\sin(\phi))\in \RR^2$ is the unit vector  associated with the projection angle $\phi$ and $\vartheta_\phi^\perp:=(-\sin(\phi),\cos(\phi))\in \RR^2$ denotes the direction rotated by 90 degrees counterclockwise; see Figure \ref{Fig_Radon_forward_and_backprojection_illustration}.
\end{definition}

\begin{remark} \label{remark_angular_domain}
Note that other choices for the angular domain concerning $\phi$ are possible, e.g., $[-\frac{\pi}{2},\frac{\pi}{2}[$, or $[0,2\pi[$, and are also used throughout the literature. Since $L_{\phi+\pi,s}=L_{\phi,-s}$, the Radon transform possesses a symmetry property making formulations for these angular domains equivalent and the results of this paper are easily extendable to such domains.
\end{remark}

Given the domains $\imgdom$ and $\sinodom$, we will consider related $L^2$ function spaces.
\begin{definition}[$L^2$ spaces]
As is common, $L^2(\imgdom)$ denotes the set of all equivalence classes $[f]$ of (two-dimensional) Lebesgue-measurable functions $f\colon \Omega \to \RR$ such that 
\begin{itemize}
\item $f=g$ almost everywhere if and only if $g\in[f]$ (i.e., they only differ on sets of (two-dimensional) Lebesgue measure zero),
\item $\|f\|_{L^2(\imgdom)} < \infty$, where\begin{equation} \|f\|_{L^2(\imgdom)}^2 := \int_{\imgdom} |f(x)|^2\dd x.
\end{equation}
\end{itemize}
Note that $L^2(\imgdom)$ is a Hilbert space with the norm $\|\cdot\|_{L^2(\imgdom)}$. The Hilbert space $L^2(\sinodom)$ is defined completely analogously with 
\begin{equation}
\|g\|_{L^2(\sinodom)}^2 := \int_{\sinodom} |g(\phi,s)|^2\dd{(\phi,s)}.
\end{equation}

\end{definition}

The (theoretical) measurement process can be understood as granting one measurement value for each $(\phi,s)\in\sinodom$ related to line integrals along $L_{\phi,s}$, resulting in the Radon transform.
\begin{definition}[Radon transform]
The Radon transform $\Radon \colon L^2(\imgdom)\to L^2(\sinodom)$ is defined according to
\begin{align}
\label{equ_def_radon_transform}
[\Radon f](\phi,s):= \int_{\RR^2} f(x) \dd {\mathcal{H}^{1}\mres L_{\phi,s}}(x) = \int_{\RR} f(s\vartheta_\phi+t \vartheta_\phi^{\perp}) \dd t 
\end{align}
for $f\in L^2(\imgdom)$ and almost all $(\phi,s)\in \sinodom$ (where $\mathcal{H}^1\mres L_{\phi,s}$ denotes the one-dimensional Hausdorff measure restricted to $L_{\phi,s}$), i.e., a collection of line integrals.

We define the (parallel-beam) backprojection $\Radon^*\colon L^2(\sinodom)\to L^2(\imgdom)$, which,  given $g\in L^2(\sinodom)$, reads
\begin{equation} \label{equ_def_backprojection}
[\Radon^* g](x) := \int_{0}^{\pi} g(\phi,x\cdot \vartheta_\phi) \dd \phi  \qquad \text{for a.e. }x\in \imgdom.
\end{equation}
\end{definition}

It is well-known that $\Radon$ and $\Radon^*$ are continuous operators between these $L^2$ spaces. 
With the standard $L^2$ inner products $\langle f,\tilde f \rangle_{L^2(\imgdom)}= \int_{\imgdom} f(x)\tilde f(x)\dd x$ and $\langle g,\tilde g \rangle_{L^2(\sinodom)}= \int_{\sinodom} g(\phi,s)\tilde g(\phi,s)\dd (\phi,s)$, the operators $\Radon$ and $\Radon^*$ are adjoint, i.e., $\langle \Radon f, g \rangle_{L^2(\sinodom)}=\langle  f ,\Radon^*g \rangle_{L^2(\imgdom)}$ for all $f\in L^2(\imgdom)$, $g\in L^2(\sinodom)$.

\begin{figure}

\centering
\usetikzlibrary{decorations.pathreplacing}
\usetikzlibrary{positioning,patterns}
\newcommand{\myfiguresize}{0.7}

\newcommand{\uppercutoff}{3.2}
\newcommand{\lowercutoff}{-2}

\begin{tikzpicture}[scale=\myfiguresize*1.4]
    \draw[gray!20,fill] (0,0) circle(2);
    \draw[fill] (0,0) circle (0.1);

	\draw (0,1.8)   node[]{\Large $\imgdom$};

    \pgfmathsetmacro{\myangle}{70}
    \pgfmathsetmacro{\beamend}{2.4}
    \pgfmathsetmacro{\detectorside}{2}
        \pgfmathsetmacro{\detectordepth}{0.4}
	\newcommand{\mycolor}{violet}
    \pgfmathsetmacro{\s}{1}
    \draw [ultra thick, red,rotate around ={\myangle+90:(0,0)}] (-\beamend,\s)--(\beamend,\s) node [midway, below] {\large $L_{\phi,s}$};

    \draw [ultra thick, \mycolor,rotate around ={\myangle+90:(0,0)}] (0,0)--(0,\s) node [midway, right] {\large $s$};
    
        \draw [ultra thick, fill,gray!50,rotate around ={\myangle+90:(0,0)}] (\beamend,-2)rectangle(\beamend+\detectordepth,2) node[midway,xshift=-0.2cm,yshift=0.1cm,above,rotate=\myangle,black]{Detector};
        \draw [very thick,\mycolor,dashed, rotate around ={\myangle+90:(0,0)},->] (\beamend+\detectordepth/2,0) -- (\beamend+\detectordepth/2,-1.5) node [above,xshift=0.1cm]{$s$};
        \draw [fill,rotate around ={\myangle+90:(0,0)},\mycolor] (\beamend+\detectordepth/2,0) circle (0.05cm);

    \clip(-3.1,\lowercutoff) rectangle (2,\uppercutoff);

    \draw [ultra thick,rotate around ={\myangle:(0,0)},orange,->] (0,0) -- (2,0) node[midway,right,yshift=.1cm] {\large $\vartheta_\phi$};
    \draw [ultra thick,rotate around ={\myangle+90:(0,0)}, teal,->] (0,0) -- (2,0) node[midway,above] {\large $\vartheta_\phi^\perp$};

    \pgfmathsetmacro{\radius}{1}
    \draw [ultra thick,blue] (\radius,0) arc (0:\myangle:\radius) node [left,below,yshift=-0.2cm] {\large $\phi$};
    \draw [dashed ] (0,0) -- (\radius,0) ;
\end{tikzpicture} \hfill   \begin{tikzpicture}[scale=0.8]

 \begin{scope}

  \draw [fill=gray!20] (0,2) rectangle (6.28,-2);


  \end{scope}

\draw[very thick,->] (0, 0) -- (6.5, 0) node[right] {$\phi$};
  \draw[very thick,->] (0, -1) -- (0, 2.5) node[above] {$s$};
   \pgfmathtruncatemacro\myindex{14}
  \pgfmathsetmacro {\mydomain}{6.28}
  \draw[scale=1, domain=0:\mydomain, smooth, variable=\x, blue, ultra thick] plot ({\x}, {2*0.9*cos(deg(\x/2-2))}) node[midway,right,xshift=3cm,yshift=1cm] {\Large $(\phi,x\cdot \vartheta_\phi)$};
 
 \pgfmathtruncatemacro\Nangle{10}

  \node[] at (3.1,2.3) {$\sinodom$};
\end{tikzpicture}

\caption{On the left, an illustration of the used geometry with a straight line $L_{\phi,s}$  in direction $\vartheta_\phi^\perp$ with normal distance to the center $s$ (which also corresponds to the detector offset). On the right, an illustration of the backprojection, where for fixed $x$, we integrate values along a sine-shaped trajectory in the sinogram domain. This trajectory corresponds to all lines $L_{\phi,s}$ passing through $x$.}	
\label{Fig_Radon_forward_and_backprojection_illustration}

\end{figure}
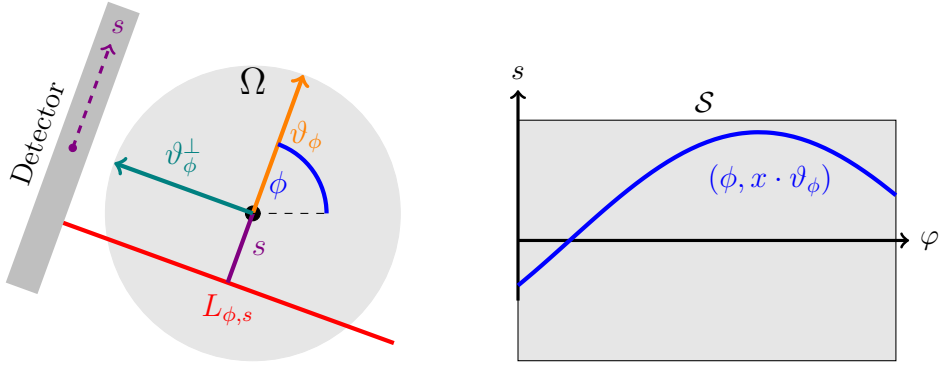

\subsection{Discretization frameworks}
\label{section_Radon_discretization}

Next, we describe the ray-driven and pixel-driven discretization frameworks as finite rank operators via convolutional discretizations. Finite rank describes operators whose range is finite-dimensional, and therefore these operators -- in essence -- also only depend on finitely many dimensions, mapping all orthogonal directions to zero \cite[II.4 Definition 4.3]{conway2019course}. 
We start by discretizing the spatial domain $\imgdom$ and the sinogram domain $\sinodom$ into `pixels'; one can think of data and reconstructions as digital images; see Figure \ref{Fig_Geometry}.

We fix $N_x\in \NN$, set $\delta_x:=\frac{2}{N_x}$, and use the notation $[N_x]:=\{0,\dots,N_x-1\}$. We define the spatial pixel centers 
\begin{multline}
x_{ij}:= \left ( \frac{2i+1}{N_x}-1,\frac{2j+1}{N_x}-1 \right)
\\
=\left(\left(i+\frac{1}{2}\right)\delta_x-1,\left (j+\frac{1}{2}\right)\delta_x-1\right) \qquad \text{for }i,j\in [N_x] 
\end{multline}
 and 
\begin{align}
 X_{ij}:=x_{ij}+\left[-\frac {\delta_x}{2},\frac{\delta_x}{2}\right]^2   \qquad { \text{for }i,j\in [N_x] }
\end{align} 
denotes the corresponding squared (spatial) pixel with side-length (resolution) $\delta_x$.

We consider a finite number of (projection) angles $\phi_0< \dots <\phi_{N_\phi-1} \in \left[0,\pi\right[$ and associate them with the angular pixels 
\begin{align}
\Phi_0 :=\left[0,\frac{\phi_{0}+\phi_1}{2}\right[ \qquad \text{and} \qquad \Phi_{N_\phi-1}:= \left[\frac{\phi_{N_\phi-2}+\phi_{N_\phi-1}}{2},\pi\right[, \notag
\\
\Phi_q := \left[\frac{\phi_{q-1}+\phi_q}{2},\frac{\phi_{q+1}+\phi_q}{2}\right[ \qquad \text{ for }q\in \{1,\dots,N_\phi-2\}.
\end{align}
Correspondingly, we set $\delta_\phi=\max_{q\in [N_\phi]} |\Phi_q|$. 
For the sake of readability, we write $\vartheta_q$ for the unit vector $\vartheta_{\phi_q}$ associated with the angle $\phi_q$.

Similarly, we assume a fixed number $N_s\in \NN$ of detector pixels and set $\delta_s=\frac{2}{N_s}$. 
The associated equispaced detector pixels are 
\begin{equation}
S_p:= s_p+\left[-\frac{\delta_s}{2},\frac{\delta_s}{2}\right[ \qquad \text{ for }p\in [N_s] 
\end{equation}
with centers 
\begin{equation}
s_p := \frac{2p+1}{N_s}-1 = \left(p+\frac{1}{2}\right) \delta_s-1 \qquad \text{ for }p\in [N_s] . 
\end{equation}

Hence, we have discretized the domain $\Omega$ (actually the larger domain $[-1,1]^2$) into a Cartesian $N_x\times N_x$ grid with (spatial) resolution $\delta_x$, while the sinogram space is discretized as an $N_\phi\times N_s$ grid of rectangular pixels $\Phi_q\times S_p$ for $q\in [N_\phi]$ and $p\in [N_s]$, i.e., with angular resolution $\delta_\phi$ and detector resolution $\delta_s$; see Figure \ref{Fig_Geometry}. We notationally combine all these resolutions to $\delta=(\delta_x,\delta_\phi,\delta_s)\in \RR^+\times\RR^+\times\RR^+$, and $N_x$, $N_\phi$ and $N_s$ are tacitly chosen accordingly.

\begin{figure}
\centering
\usetikzlibrary{decorations.pathreplacing}
\usetikzlibrary{positioning,patterns}
\newcommand{\myfiguresize}{0.77}

\newcommand{\uppercutoff}{3.2}
\newcommand{\lowercutoff}{-3}
\begin{tikzpicture}[scale=\myfiguresize]
  \draw[thick,fill, gray!20] (0,0) circle (2 );

\clip(-3.2,\lowercutoff) rectangle (2.8,\uppercutoff);
\pgfmathsetmacro{\N}{6}
\foreach \n in {0,...,\N}
{
\pgfmathsetmacro{\s}{2*(\n/\N*2-1)}
\draw[] (-2,\s) -- (2,\s);
\draw[] (\s,-2) -- (\s,2);
}

\pgfmathsetmacro{\mydelta}{(1/\N*2)}
\foreach \n in {1,...,\N}
{
\pgfmathsetmacro{\s}{2*((\n-1/2)/\N*2-1)}
\draw[teal] (\s,2-\mydelta) -- (0,2.5) node [above] {$N_x$};
\draw[fill,teal] (\s,2-\mydelta) circle (0.05cm);

\draw[magenta] (-2+\mydelta,\s) -- (-2.4,0) node [left,xshift=0.cm] {$N_x$};
\draw[fill,magenta] (-2+\mydelta,\s) circle (0.05cm);
}

\draw[,decorate,decoration={brace,amplitude=3pt,mirror},violet] (2,2-2*\mydelta) -- (2,2) node[midway,right] {$\delta_x$};
\draw[,decorate,decoration={brace,amplitude=3pt,mirror},violet] (2,2) -- (2-2*\mydelta,2) node[midway,above] {$\delta_x$};

\pgfmathsetmacro{\n}{2}
\pgfmathsetmacro{\m}{1}
\pgfmathsetmacro{\s}{2*((\n-1/2)/\N*2-1)}
\pgfmathsetmacro{\ss}{2*((\m-1/2)/\N*2-1)}

\draw[fill=blue] (\s-2*1/\N,\ss-2*1/\N) rectangle (\s+2*1/\N,\ss+2*1/\N);
\draw [dashed, blue] (\s,\ss) -- (\s+0.5,\ss-0.5)  node[below]{\footnotesize $X_{ij}$};

\draw [dashed, red] (\s,\ss) -- (\s-0.5,\ss-0.5)  node[below]{\footnotesize $x_{ij}$};
\draw [fill,red] (\s,\ss) circle (0.05);

\node[orange] at (0,0) {\Large $\imgdom$};

\end{tikzpicture} \qquad 
\begin{tikzpicture}[scale=0.7]
\pgfmathsetmacro{\startpointx}{-3}
\pgfmathsetmacro{\startpointy}{-3}

\draw [fill,gray!10] (-2,-2) rectangle (2,2);

\draw[dashed,->] (\startpointx,\startpointy) -- (\startpointx,2.5) node[midway,rotate=270,yshift=0.2cm]{\footnotesize Angle};
\draw[dashed,->] (\startpointx,-3) -- (2.6,-3)  node[midway,yshift=-0.2cm]{\footnotesize Detector};
\draw[thick] (\startpointx-0.2,2) --(\startpointx+0.2,2) node [xshift=0.3cm]{$\footnotesize \pi$};
\draw[thick] (\startpointx-0.2,-2) --(\startpointx+0.2,-2) node [xshift=0.3cm]{$\footnotesize0$};
\draw[thick] (-2,-3+0.2) -- (-2,-3-0.2) node[yshift=-0.1cm]{$\footnotesize -1$};
\draw[thick] (2,-3+0.2) -- (2,-3-0.2) node[yshift=-0.1cm,xshift=0.1cm]{ $\footnotesize 1$};

\pgfmathsetmacro{\N}{6}
\foreach \n in {0,...,\N}
{
\pgfmathsetmacro{\s}{2*(2*\n/\N-1)}
\draw (-2,\s) -- (2,\s);
\draw (\s,-2) -- (\s,2);

}
\foreach \n in {1,...,\N}
{
\pgfmathsetmacro{\s}{2*(2*\n/\N-1)}
\draw[teal](\s-2/\N, -2+2/\N)-- (0,-2.2) node[below]{$N_s$};
\draw[fill,teal](\s-2/\N, -2+2/\N) circle (0.05);

\draw[magenta](2-2/\N,\s-2/\N)-- (2.2,0) node[right]{$N_\phi$};
\draw[fill,magenta](2-2/\N,\s-2/\N) circle (0.05);

}

\pgfmathsetmacro{\s}{6/\N}
\pgfmathsetmacro{\ss}{2-2/\N}

\draw[fill=blue](\s-2/\N,\ss-2/\N) rectangle (\s+2/\N,\ss+2/\N);
\draw[dashed,blue] (\s,\ss) -- (\s-1,\ss+0.5) node[above]{$\Phi_q \times S_p$};

\draw[dashed,red] (\s,\ss) -- (\s+1,\ss+0.5) node[above]{$(\phi_q,s_p)$};
\draw [fill,red] (\s,\ss)  circle (0.05);

\draw [ultra thick,blue] (\s-2/\N,-3) -- (\s+2/\N,-3) node[midway,below]{$S_p$};
\draw [ultra thick,blue] (\startpointx,\ss-2/\N) -- (\startpointx,\ss+2/\N) node[midway,left]{$\Phi_q$};

\draw [decorate,decoration={brace,amplitude=5pt,mirror},violet] (2,-2) -- (2,-2+4/\N) node[midway,right,xshift=.1cm] {$|\Phi_q|\leq \delta_\phi$};

\draw [decorate,decoration={brace,amplitude=5pt,mirror},violet] (2-4/\N,-2) -- (2,-2) node[midway,below, yshift=-.1cm] {$\delta_s$};

\node[orange] at (0,0) {\Large $\sinodom$};

\end{tikzpicture}

    \caption{On the left, the spatial domain  $\imgdom$ (in fact, the larger domain $[-1,1]^2$) is divided into pixels $X_{ij}$ with width $\delta_x\times \delta_x$. On the right, the discretization of the sinogram domain $\sinodom$ into pixels $\Phi_q\times S_p$ with pixel centers $(\phi_q,s_p)$ and  with width $|\Phi_q|\times \delta_s$ is shown.}
    \label{Fig_Geometry}
\end{figure}
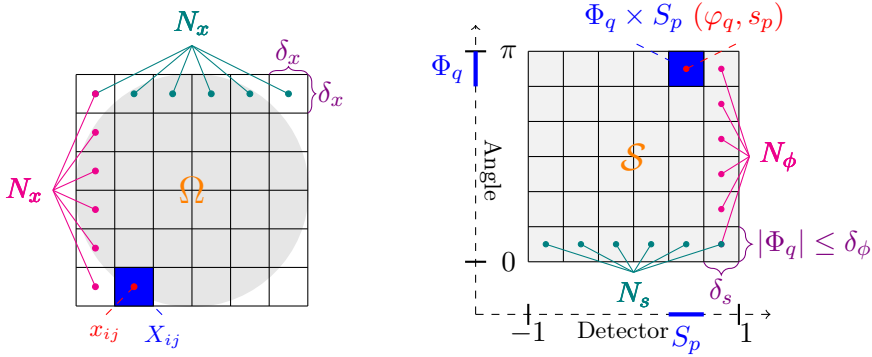

One can naturally associate  pixel values of an image representing $f\in L^2(\imgdom)$ with average values $f_{ij}:=\frac{1}{\delta_x^2}\int_{X_{ij}}f(x) \dd x$ and corresponding piecewise constant functions $f_\delta = \sumx f_{ij} u_{ij}$ in $U_\delta:=\text{span}\{u_{ij}\}_{i,j\in[N_x]}\widehat = \RR^{N_x^2}$ with $u_{ij}:=\chi_{X_{ij}}-\frac 1 2 \chi_{\partial X_{ij}}$ where $\chi_M(x)$ equals 1 if $x\in M$  and zero otherwise, and  $\partial X_{ij}$ denotes the boundary of $X_{ij}$. (In other words, $u_{ij}$ attains the value 1 inside $X_{ij}$, $\frac{1}{2}$ on its boundary and zero otherwise.)
Similarly, we can consider sinogram images as functions $g_\delta \in V_\delta := \text{span}\{v_{qp}\}_{q\in [N_\phi],p\in [N_s]}\widehat = \RR^{N_\phi\cdot N_s}$ with $v_{qp}:=\chi_{\Phi_q\times S_p}$ and the associated coefficients $g_{qp}$ are again average values on pixels.

Discretizations of $\Radon$ translate to a matrix-vector multiplication with the matrix $A\in \RR^{ (N_\phi \cdot N_s) \times N_x^2 }$ mapping from $U_\delta$ to $V_\delta$ (we think of the collection of pixel values $(f_{ij})_{ij}$ and $(g_{qp})_{qp}$ as vectors $\overline f$ and $\overline g$ with entries $\overline f [ij]=f_{ij}$). In practical implementations, these matrices are rarely saved (due to memory constraints). Rather, matrix-free formulations are employed, i.e., the relevant matrix entries are calculated when needed and discarded afterward. 
The matrix entries $A[qp,ij]$ (the combination of $q$ and $p$ determines a row, while $i$ and $j$ determine a column) represent the weight attributed to a pixel $X_{ij}$ in the calculation for $L_{\phi_q,s_p}$. In order to compute $[\Radon f](\phi,s)$ for a specific pair $(\phi,s)$, one has to take relatively few values (the values along $L_{\phi,s}$) into account; therefore, also the $A[qp,ij]$ should be non-zero only for pixels that are close to $L_{\phi_q,s_p}$. This way, the matrix $A$ is relatively sparse, which is of practical importance.

In order to structurally analyze the discretizations, it is necessary to understand what individual matrix entries are. To that end, one realizes that for the ray-driven and pixel-driven methods, the matrix entries are evaluations of relatively simple and interpretable weight functions (in closed form) at geometrically meaningful positions.
\begin{definition}[Weight functions]
Given $\delta$ and $\phi\in [0,\pi[$, we set\\ $\overline s(\phi) := \frac{\delta_x}{2} (| \cos(\phi)|+| \sin(\phi)|)$, $\underline s(\phi) := \frac{\delta_x}{2} (\big | |\cos(\phi)|- |\sin(\phi)|\big|)$ and   $\kappa(\phi) :=  \min \left\{\frac{1}{|\cos(\phi)|},\frac{1}{|\sin(\phi)|}\right\}$. We define the ray-driven weight function for $t\in \RR$ according to
\begin{equation}\label{equ_def_ray_weight}
\RayWeight(\phi,t):= \frac{1}{\delta_x}
\begin{cases} 
\kappa(\phi) \qquad &\text{if } |t| <\underline s(\phi),
\\
 \frac{\overline s(\phi)-|t|}{\delta_x|\cos(\phi)\sin(\phi)|} \qquad & \text{if } |t| \in [\underline s(\phi),\overline s(\phi)[,
\\
\frac{1}{2} \qquad & \text{if } \phi\in \frac{\pi}{2}\ZZ \text{ and } |t| =\overline s (\phi),
\\
0 & \text{else},	
\end{cases}
\end{equation}
where $\frac{\pi}{2}\ZZ$ denotes all multiples of $\frac{\pi}{2}$. 
Moreover, we define the pixel-driven weight function to be
\begin{equation}\label{equ_def_pixel_weight}
\PixelWeight(\phi,t)=\PixelWeight(t):= \frac{1}{\delta_s^2} \max\{\delta_s-|t|,0\} \qquad \text{for } t\in \RR, \  \phi\in \left [0,\pi\right[.
\end{equation}
\end{definition}

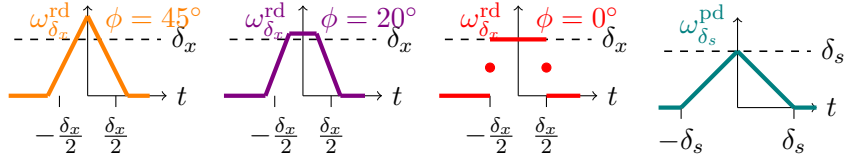
\begin{figure}
\centering
    \newcommand{\mywidthplot}{1.4}
    \newcommand{\myscale}{0.75}
    \newcommand{\myradius}{0.08}
    \newcommand{\mylinewidth}{ultra thick}
    \newcommand{\mycolor}{red}
    \newcommand{\mytextpositionone}{-0.8}
    \newcommand{\mytextpositiontwo}{1}
    \newcommand{\myarrowup}{1.6}
\begin{tikzpicture}
    \newcommand{\mycolorthree}{orange}
    
\pgfmathsetmacro{\myangle}{45}
\pgfmathsetmacro{\supper}{(cos(\myangle)+sin(\myangle))/2}
\pgfmathsetmacro{\sunder}{(-cos(\myangle)+sin(\myangle))/2}
\pgfmathsetmacro{\mykappa}{1/sin(\myangle)}

        \draw[scale=\myscale,->] (0,0) -- (\mywidthplot,0) node[right] {$t$};
        \draw[scale=\myscale,->] (0,0) -- (0,\myarrowup) ;

        \draw[scale=\myscale] (0.5,0) --(0.5,-0.2) node[below]{$\frac{\delta_x}{2}$};
        \draw[scale=\myscale] (-0.5,0) --(-0.5,-0.2) node[below]{$-\frac{\delta_x}{2}$};
        \draw[dashed, semithick, scale=\myscale] (-\mywidthplot+0.1,1) --(\mywidthplot-0.1,1) node[right]{$\delta_x$};
    
\draw[scale=\myscale, domain=-\mywidthplot:-\supper, smooth, variable=\x, \mycolorthree,\mylinewidth] plot ({\x}, {0});
\draw[scale=\myscale, domain=-\sunder:\sunder, smooth, variable=\x, \mycolorthree,\mylinewidth] plot ({\x}, {\mykappa});
\draw[scale=\myscale, domain=-\supper:-\sunder, smooth, variable=\x, \mycolorthree,\mylinewidth] plot ({\x}, {\mykappa*(\x+\supper)/(\supper-\sunder)});
\draw[scale=\myscale, domain=\sunder:\supper, smooth, variable=\x, \mycolorthree,\mylinewidth] plot ({\x}, {\mykappa*(-\x+\supper)/(\supper-\sunder)});
\draw[scale=\myscale, domain=\supper:\mywidthplot-0.3, smooth, variable=\x, \mycolorthree,\mylinewidth] plot ({\x}, {0});

    \draw[\mycolorthree] (\mytextpositionone,\mytextpositiontwo) node[right,xshift=-0.1cm] {$\RayWeight \quad \phi=45^\circ$};
\end{tikzpicture}
\begin{tikzpicture}
    \newcommand{\mycolorthree}{violet}
\pgfmathsetmacro{\myangle}{65}
\pgfmathsetmacro{\supper}{(cos(\myangle)+sin(\myangle))/2}
\pgfmathsetmacro{\sunder}{(-cos(\myangle)+sin(\myangle))/2}
\pgfmathsetmacro{\mykappa}{1/sin(\myangle)}

        \draw[scale=\myscale,->] (0,0) -- (\mywidthplot,0) node[right] {$t$};
        \draw[scale=\myscale,->] (0,0) -- (0,\myarrowup) ;

        \draw[scale=\myscale] (0.5,0) --(0.5,-0.2) node[below]{$\frac{\delta_x}{2}$};
        \draw[scale=\myscale] (-0.5,0) --(-0.5,-0.2) node[below]{$-\frac{\delta_x}{2}$};
        \draw[dashed, semithick, scale=\myscale] (-\mywidthplot+0.1,1) --(\mywidthplot-0.1,1) node[right]{$\delta_x$};
    
\draw[scale=\myscale, domain=-\mywidthplot:-\supper, smooth, variable=\x, \mycolorthree,\mylinewidth] plot ({\x}, {0});
\draw[scale=\myscale, domain=-\sunder:\sunder, smooth, variable=\x, \mycolorthree,\mylinewidth] plot ({\x}, {\mykappa});
\draw[scale=\myscale, domain=-\supper:-\sunder, smooth, variable=\x, \mycolorthree,\mylinewidth] plot ({\x}, {\mykappa*(\x+\supper)/(\supper-\sunder)});
\draw[scale=\myscale, domain=\sunder:\supper, smooth, variable=\x, \mycolorthree,\mylinewidth] plot ({\x}, {\mykappa*(-\x+\supper)/(\supper-\sunder)});
\draw[scale=\myscale, domain=\supper:\mywidthplot-0.3, smooth, variable=\x, \mycolorthree,\mylinewidth] plot ({\x}, {0});

\draw[\mycolorthree] (\mytextpositionone,\mytextpositiontwo) node[right,xshift=-0.1cm] {$\RayWeight \quad \phi=20^\circ$};
\end{tikzpicture}
    \begin{tikzpicture}

        \draw[scale=\myscale,->] (0,0) -- (\mywidthplot,0) node[right] {$t$};
        \draw[scale=\myscale,->] (0,0) -- (0,\myarrowup) ;

        \draw[scale=\myscale] (0.5,0) --(0.5,-0.2) node[below]{$\frac{\delta_x}{2}$};
        \draw[scale=\myscale] (-0.5,0) --(-0.5,-0.2) node[below]{$-\frac{\delta_x}{2}$};
        \draw[dashed, semithick, scale=\myscale] (-\mywidthplot+0.1,1) --(\mywidthplot-0.1,1) node[right]{$\delta_x$};

        \draw[scale=\myscale, domain=-\mywidthplot:-0.5, smooth, variable=\x, \mycolor,\mylinewidth] plot ({\x}, {0});
        \draw[scale=\myscale,fill,\mycolor] (-0.5,0.5) circle (\myradius cm);
        \draw[scale=\myscale, domain=-0.5:0.5, smooth, variable=\x, \mycolor,\mylinewidth] plot ({\x}, {1});
        \draw[scale=\myscale,fill,\mycolor] (0.5,0.5) circle (\myradius cm);
        \draw[scale=\myscale, domain=0.5:\mywidthplot-0.3, smooth, variable=\x, \mycolor,\mylinewidth] plot ({\x}, {0});

        \draw[\mycolor] (\mytextpositionone,\mytextpositiontwo) node[right,xshift=-0.3cm] {$\RayWeight \quad \phi=0^\circ$};;
        \end{tikzpicture}
\begin{tikzpicture}
    \newcommand{\mycolorfour}{teal}

        \draw[scale=\myscale,->] (0,0) -- (\mywidthplot,0) node[right] {$t$};
        \draw[scale=\myscale,->] (0,0) -- (0,\myarrowup) ;

        \draw[scale=\myscale] (1,0) --(1,-0.2) node[below]{$\delta_s$};
        \draw[scale=\myscale] (-1,0) --(-1,-0.2) node[below]{$-\delta_s$};
        \draw[dashed, semithick, scale=\myscale] (-\mywidthplot+0.1,1) --(\mywidthplot-0.1,1) node[right]{$\delta_s$};
    
\draw[scale=\myscale, domain=-\mywidthplot:-1, smooth, variable=\x, \mycolorfour,\mylinewidth] plot ({\x}, {0});
\draw[scale=\myscale, domain=-1:0, smooth, variable=\x, \mycolorfour,\mylinewidth] plot ({\x}, {1+\x});
\draw[scale=\myscale, domain=-0:1, smooth, variable=\x, \mycolorfour,\mylinewidth] plot ({\x}, {1-\x});

\draw[scale=\myscale, domain=1:\mywidthplot, smooth, variable=\x, \mycolorfour,\mylinewidth] plot ({\x}, {0});

    \draw[\mycolorfour] (\mytextpositionone,\mytextpositiontwo) node [right,xshift=-0.1cm,yshift=0.1cm] {$\PixelWeight$};
\end{tikzpicture}

\caption{Depiction of the ray-driven weight function $t\mapsto \delta_x^2\RayWeight(\phi,t)$ for fixed $\phi\in \{0^\circ,20^\circ,45^\circ\}$ in the first three plots. For fixed $\phi$, these are trapezoid functions (like the $20^\circ$ case), whose incline, height, and width depend on $\phi$ and $\delta_x$. In the extreme case $\phi=45^\circ$, the function turns into a hat function, while for $\phi=0^\circ$ it turns into a piecewise constant function (note the values for $\pm \frac{\delta_x}{2}$). On the right, in the last plot, the pixel-driven weight function $t\mapsto \delta_s^2\PixelWeight(t)$ (a hat function) independent of $\phi$ is shown. Note the difference in scales between the ray-driven and pixel-driven functions.}
\label{Fig_graphs_ray_driven}

\end{figure}
Note that for fixed $\phi\not \in  \frac{\pi}{2}\ZZ$, the map $t\mapsto \RayWeight(\phi,t)$ has a trapezoidal structure, see Figure \ref{Fig_graphs_ray_driven}, with the first case of \eqref{equ_def_ray_weight} relating to the constant inner part, the second case connecting this constant part with zero in a continuous ramp and the fourth condition reflecting zero values outside.
But when $\phi\in  \frac{\pi}{2}\ZZ$, the second case in \eqref{equ_def_ray_weight} is empty (since $\underline s(\phi)=\overline s(\phi)$) and thus a discontinuity appears when one wants the function to be constant ($\kappa(\phi)=1$) inside and zero for $|t|>\overline s$. To bridge this discontinuity, the third case relates to half the jump height on the boundary points; see Figure \ref{Fig_graphs_ray_driven}.

The ray-driven method (as described in the literature) uses the intersection lengths of lines and pixels as weights, i.e., $A[qp,ij] = \mathcal{H}^1(L_{\phi_q,s_p}\cap X_{ij})$ (again, $\mathcal{H}^1$ denotes the one-dimensional Hausdorff measure). These are computed in an iterative manner following the ray and calculating the next intersection of the ray with the pixel grid from previous positions in a finite-step search  \cite{doi:10.1118/1.4761867}; see Figure \ref{Fig_ray_driven_interpretation}. The value $\delta_x^2 \RayWeight(\phi_q,x_{ij}\cdot\vartheta_q-s_p)$ is a closed-form expression of this weight (see Lemma \ref{Lemma_weight_as_intersection_length} below) required for the more structured analysis we will execute in Section \ref{section_Radon_proofs}. The special case $\phi\in \frac{\pi}{2}\ZZ \text{ and } |t| =\overline s (\phi)$ in \eqref{equ_def_ray_weight} relates to when $L_{\phi,s}\cap X_{ij} = L_{\phi,s}\cap \partial X_{ij}$ is one side of the pixel $X_{ij}$. To avoid counting said edge twice (once for each of the pixels containing the edge), we attribute half the intersection length to either of the two pixels sharing this side. This choice is somewhat arbitrary; what matters is that they sum up to $1$.

The pixel-driven weight is such that there are at most two $p$ (for fixed $q\in [N_\phi]$ and $i,j\in [N_x]$) such that $A[qp,ij]\neq 0$, and whose sum equals 1 (see Lemma \ref{Lemma_weight_length}). One can imagine the pixel's contribution is distributed onto the two closest lines; one speaks of anterpolation. Moreover, this results in a backprojection with linear interpolation (with respect to the detector dimension) of the closest relevant detector pixels; see Figure \ref{Fig_ray_driven_interpretation}.

\begin{lemma}[Closed form of the intersection length]
\label{Lemma_weight_as_intersection_length}
Given $\delta$, $\phi \in [0,\pi[$ and $s \in \RR$,
we have 
\begin{equation} \label{equ_lemma_intersectionlength}
\delta_x^2 \RayWeight(\phi,x_{ij}\cdot \vartheta_\phi - s)= \mathcal{H}^1(L_{\phi,s}\cap X_{ij}) - \frac 1 2 \mathcal{H}^1(L_{\phi,s}\cap \partial X_{ij}) .
\end{equation}
\end{lemma}
The proof of this statement is quite geometric with multiple case distinctions and is found in the Appendix.

In order to compare the matrices representing discretizations with $\Radon$ and $\Radon^*$, we next reinterpret them as finite rank operators mapping from $L^2(\imgdom)$ to $L^2(\sinodom)$ or vice versa. More precisely, they map into $U_\delta$ and $V_\delta$ spanned by $u_{ij}:=\chi_{X_{ij}}-\frac{1}{2}\chi_{\partial X_{ij}}$ and $v_{qp}:=\chi_{\Phi_q\times S_p}$ for $i,j\in [N_x]$, $q\in [N_\phi]$ and $p\in [N_s]$, respectively.
\begin{definition} [Convolutional discretizations]
	We choose the ray-driven or pixel-driven setting by choosing $\omega \in \{\RayWeight,\PixelWeight\}$.
    Given $\delta$, the ray-driven Radon transform $\RayRadon$ and the pixel-driven Radon transform $\PixelRadon$, respectively, are defined via $\GenRadon\colon L^2(\imgdom)\to L^2(\sinodom)$, such that, for a function $f\in L^2(\imgdom)$, 
    \begin{equation}\label{equ_def_convolutional_Radon_transform}
[\GenRadon f](\phi,s) := \sumphi\sums v_{qp}(\phi,s) \sumx \omega (\phi_q,x_{ij}\cdot \vartheta_q -s_p) \int_{X_{ij}} f(x) \dd x.
    \end{equation}

    The corresponding ray-driven or pixel-driven backprojections $\RayRadon^*$ and $\PixelRadon^*$ respectively are  defined via $\GenRadon^*\colon L^2(\sinodom)\to L^2(\imgdom)$ according to
    \begin{equation} \label{equ_def_convolution_backprojection}
        [\GenRadon^* g](x) : =   \sumx u_{ij}(x) \sumphi\sums \omega(\phi_q,x_{ij}\cdot \vartheta_q -s_p) \int_{\Phi_q\times S_p} g(\phi,s) \dd{(\phi,s)}
    \end{equation}
    for $g\in L^2(\sinodom)$ and $x\in \imgdom$. 
\end{definition}



 Note that the output of these operators is constant on the pixels, thus mapping $U_\delta$ to $V_\delta$ or vice-versa. Let $\overline f \in \RR^{N_x^2}$ and $\overline g\in \RR^{N_\phi\cdot N_s}$ be vectors whose entries coincide with the coefficients $f_{ij}=\frac{1}{\delta_x^2}\int_{X_{ij}}f \dd x$ and $g_{qp}=\frac{1}{|\Phi_q|\delta_x}\int_{\Phi_q\times S_p}g \dd{(\phi,s)}$, i.e., $\overline{f}$ and $\overline{g}$ are the coefficient vectors of $f$ and $g$. For these vectors, the matrices $A\approx \GenRadon$ and $B\approx \GenRadon^*$ perform
\begin{align} 
&(A \overline f)[{qp}] = \delta_x^2\sumx \omega(\phi_q,x_{ij}\cdot \vartheta_q-s_p) \overline f[{ij}] , \label{equ_def_convolutional_Radon_discrete}
\\
  &(B \overline g)[{ij}] = \delta_s \sumphi |\Phi_q| \sums \omega(\phi_q,x_{ij}\cdot \vartheta_q-s_p) \overline g[{qp}], \label{equ_def_convolutional_backproj_discrete}
\end{align}
i.e., $A[qp,ij]= \delta_x^2  \omega(\phi_q,x_{ij}\cdot \vartheta_q-s_p)$ and $B[ij,qp]= \delta_s |\Phi_q| \omega(\phi_q,x_{ij}\cdot \vartheta_q-s_p)$. In particular, if $|\Phi_q|=\delta_\phi$ constant, $A^T = \frac{\delta_x^2}{\delta_\phi\delta_s} B$. So these matrices are also adjoint in a discrete sense. The different prefactors relate to the scaling in $U_\delta$ and $V_\delta$ (rather than $\RR^{N_x^2}$ and $\RR^{N_\phi \cdot N_s}$) and has nothing to do with an adjoint mismatch, but rather is the native scaling for these operators.
Plugging $\RayWeight$ and $\PixelWeight$ in, these matrix multiplications coincide (up to scaling) with the classical definitions of the ray-driven and pixel-driven methods.

\begin{figure}
\centering

\begin{tikzpicture}[scale=0.8]

 \definecolor{light_gray}{gray}{0.9}
\definecolor{light_gray2}{gray}{0.7}
\pgfmathsetmacro {\myangle}{30}
\pgfmathsetmacro {\myRE}{3.5}
\pgfmathsetmacro {\myRD}{3.}
\pgfmathsetmacro {\myDW}{2}
\pgfmathsetmacro {\Ns}{6} 
\pgfmathsetmacro {\N}{5}

\pgfmathsetmacro {\myposition}{\Ns/2-2}
\pgfmathsetmacro {\myx}{\myRD+\myRE}
\pgfmathsetmacro {\myy}{-\myposition/\Ns*\myDW*2+\myDW-1/\Ns*\myDW}
\pgfmathsetmacro {\mypositionm}{\myposition-1}
\pgfmathsetmacro {\mypositionp}{\myposition+1}

\pgfmathsetmacro {\myyplus}{-\mypositionp/\Ns*\myDW*2+\myDW-1/\Ns*\myDW}
\pgfmathsetmacro {\myyminus}{-\mypositionm/\Ns*\myDW*2+\myDW-1/\Ns*\myDW}

\pgfmathsetmacro {\mycos}{cos(\myangle+90)};
\pgfmathsetmacro {\mysin}{sin(\myangle+90)};

\pgfmathsetmacro {\mynewx}{\mycos*\myx};
\pgfmathsetmacro {\mynewy}{\mysin*\myx};

\pgfmathsetmacro {\mycos}{cos(\myangle)};
\pgfmathsetmacro {\mysin}{sin(\myangle)};

\pgfmathsetmacro {\sourcex}{\mycos*(\myRE)-\mysin*(\myy)};
\pgfmathsetmacro {\sourcey}{\mysin*(\myRE)+ \mycos*\myy};

\pgfmathsetmacro {\mynorm}{sqrt(\mynewx*\mynewx+\mynewy*\mynewy)};
\pgfmathsetmacro {\mynewx}{\mynewx/\mynorm};
\pgfmathsetmacro {\mynewy}{\mynewy/\mynorm};

\pgfmathsetmacro {\projectionx}{\mynewy};
\pgfmathsetmacro {\projectiony}{-\mynewx};

\foreach \x in {0,...,\N}
{
\draw[black,thick] (4/\N*\x-2,-2) -- (4/\N*\x-2,2);
\draw[black,thick] (-2,4/\N*\x-2) -- (2,4/\N*\x-2);
}

\foreach \x in {1,...,\N}
{
\foreach \y in {1,...,\N}
{
\pgfmathsetmacro {\posx}{4/\N*\x-2-2/\N};
\pgfmathsetmacro {\posy}{4/\N*\y-2-2/\N};

\pgfmathsetmacro {\posxnew}{\posx-\sourcex};
\pgfmathsetmacro {\posynew}{\posy-\sourcey};

\pgfmathsetmacro {\valp}{\posxnew*\mynewx+\posynew*\mynewy};

}
}
\draw[thick] (0,0) circle (2cm);

\draw[very thick,dashed, magenta,rotate around={\myangle:(0,0)}] (-\myRE,\myy) -- (\myRD+0.4,\myy) node [midway,cyan,xshift=-0.125cm,yshift=-0.2cm] {};

\foreach \x/\y/\mycolor in {0/2/orange,0/1/yellow,1/1/violet,2/1/cyan,2/0/purple,3/0/green,4/0/red}
{
\pgfmathsetmacro {\myupper}{2-\y/\N*4};
\pgfmathsetmacro {\mylower}{2-\y/\N*4-4/\N};
\pgfmathsetmacro {\myleft}{-2+\x/\N*4};
\pgfmathsetmacro {\myright}{-2+\x/\N*4+4/\N};

\begin{scope}
 \clip[] (\myleft,\mylower) rectangle (\myright,\myupper);
 \draw[ultra thick, \mycolor,rotate around={\myangle:(0,0)}] (-\myRE,\myy) -- (\myRD+0.4,\myy) node [midway,cyan,xshift=-0.125cm,yshift=-0.2cm] {};
 \end{scope}
}

\draw [very thick,fill=light_gray,rotate around={\myangle:(0,0)}] (\myRD,-\myDW) rectangle (\myRD+0.4,\myDW) node[midway,above,rotate=\myangle-90,yshift=0.1cm]{Detector};


\foreach \s in {1,...,\Ns}
{
\draw [rotate around={\myangle:(0,0)}] (\myRD,\s/\Ns*\myDW*2-\myDW) -- (\myRD+0.4,\s/\Ns*\myDW*2-\myDW) ;

}


\draw [very thick,dashed,fill=magenta!30!white,rotate around={\myangle:(0,0)}] (\myRD,-\myposition/\Ns*\myDW*2+\myDW-0/\Ns*\myDW*2) rectangle (\myRD+0.4,-\myposition/\Ns*\myDW*2+\myDW-1/\Ns*\myDW*2);
  \end{tikzpicture}
  \qquad
  \begin{tikzpicture}[scale=0.75]

  \draw [fill=gray!20] (0,2) rectangle (6.28,-2);

\draw[->] (0, 0) -- (6.5, 0) node[right] {$\phi$};
  \draw[->] (0, -1) -- (0, 2.5) node[above] {$s$};
   \pgfmathtruncatemacro\myindex{14}
  \pgfmathsetmacro {\mydomain}{6.28}
  \draw[scale=1, domain=0:\mydomain, smooth, variable=\x, violet, ultra thick] plot ({\x}, {2*0.8*cos(deg(\x/2-2))});
 
 \pgfmathtruncatemacro\Nangle{10}
  \foreach \i in {0,...,\Nangle}
  {
    \pgfmathsetmacro {\x}{6.28*\i/\Nangle}
  \draw[ dashed ] (\x,-2) -- (\x,2);
  }

 \pgfmathtruncatemacro\Ns{4}
  \foreach \i in {0,...,\Ns}
  {
    \pgfmathsetmacro {\x}{2*(2*\i/\Ns-1)}
  \draw[ dashed ] (0,\x) -- (6.28,\x);
  }

\foreach \i in {1,...,\Nangle} 
 {
     \pgfmathsetmacro {\x}{6.28*\i/\Nangle-6.28/\Nangle/2}
 \draw[ mark=x,mark size =3,scale=1, domain=\x:\x+0.01, smooth, variable=\x, cyan, thick] plot ({\x}, {2*0.8*cos(deg(\x/2-2))});
 }

 \begin{scope}

\clip (0,2) rectangle (6.28,-2);

 \foreach \i/\j in {1/2,2/2,3/3,4/3,5/3,6/4,7/4,8/3,9/3,10/3} 
 {
     \pgfmathsetmacro {\x}{6.28*\i/\Nangle-6.28/\Nangle/2}
   \pgfmathsetmacro {\y}{2*(2*\j/\Ns-1-1/\Ns)}
 \pgfmathsetmacro {\z}{2*(2*\j/\Ns-1+1/\Ns)}

\draw[ mark=x,mark size =3,scale=1, domain=\x:\x+0.01, smooth, variable=\x, orange, thick] plot ({\x}, {\y}); 
 
 \draw[ mark=x,mark size =3, domain=\x:\x+0.01, smooth, variable=\x, orange, thick] plot ({\x}, {\z});
 }
 
  \end{scope}
  \node[] at (3.1,2.3) {$\sinodom$};
\end{tikzpicture}
    \caption{Illustration of the ray-driven forward (left) and the pixel-driven backprojection (right). The ray-driven method splits integration along a straight line into the sum of values on pixels times their intersection length (colored segments). The pixel-driven backprojection approximates the angular integral \eqref{equ_def_backprojection} (along the violet curve $x\cdot\vartheta_\phi$) by a finite sum (Riemann sum) of angular evaluations $x\cdot\vartheta_q$ (the cyan crosses), whose values are approximated via linear interpolation in the detector dimension (using the neighboring orange pixel centers).}
    \label{Fig_ray_driven_interpretation}
\end{figure}
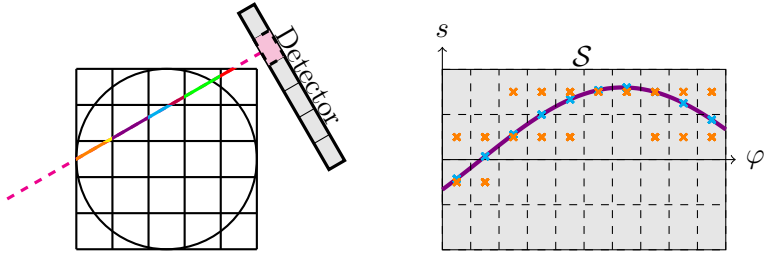

\begin{remark} \label{remark_using_hausdorff_ae}
Obviously, $u_{ij}=\chi_{X_{ij}}- \frac{1}{2} \chi_{\partial X_{ij}}=\chi_{X_{ij}}$ in a (Lebesgue) almost everywhere sense. Thus, also $\Radon (\chi_{X_{ij}}- \frac{1}{2} \chi_{X_{ij}})= \Radon \chi_{X_{ij}}$ almost everywhere and values of Lebesgue null sets are irrelevant.
However, if we evaluate the Radon transform pointwise for discretization purposes, suddenly Lebesgue null sets in $\imgdom$ and $\sinodom$ could be relevant; we are particularly concerned with the angles in $\frac{\pi}{2}\ZZ$ and $|x_{ij}\cdot \vartheta_q-s_p|=\underline s(\phi)=\delta_x$. Rather, due to the Radon transform's relation to the one-dimensional Hausdorff measure $\mathcal{H}^1$, we define $u_{ij}$ in an $\mathcal{H}^1$ almost everywhere sense, with the value $1$ in the interior of $X_{ij}$, $\frac{1}{2}$ on the boundary. The corners are a bit of a special case as we would actually like the values to be $\frac{1}{4}$, but these are an $\mathcal{H}^1$ null set and thus of no consequence. Analogously, functions $f_\delta\in U_\delta$ are understood as being defined $\mathcal{H}^1$ almost everywhere. Note that $U_\delta$ is not a native subset of $L^2(\imgdom)$, but of $L^2([-1,1]^2)$. We tacitly extend the definition of the Radon transform in \eqref{equ_def_radon_transform} to $L^2([-1,1]^2)$ where necessary. One could make analogous constructions for the $v_{qp}$ basis being defined $\mathcal{H}^1$ almost everywhere, however, these will not be necessary for our analysis below.

\end{remark}

\section{Convergence Analysis}
\label{section_Radon_proofs}

\subsection{Formulation  of convergence results}
\label{section_results}
These discretization frameworks have been known for decades (at least as practical implementations), but no rigorous convergence analysis was conducted. Our interpretation of discretizations as finite rank operators (via convolutional discretizations) allows the comparison of `continuous' and `discrete' operators. Those comparisons culminate in Theorem \ref{Thm_approximation_ray_driven}, which complements anecdotal reports on the performance of discretization approaches by describing convergence in the strong operator topology (SOT).

\begin{theorem}[Convergence in the strong operator topology] \label{Thm_approximation_ray_driven} \ \\
    Let $(\delta^n)_{n\in \NN}=(\delta_x^n,\delta_\phi^n,\delta_s^n)_{n\in \NN}$ be a  sequence of discretization parameters with $\delta^n\overset{n\to \infty}{\to} 0$ (componentwise) and let $c>0$ be a constant.
    \\ If $\frac{\delta_s^n}{\delta_x^n} \leq c$ for all $n\in \NN$, then,  for any $f\in L^2(\imgdom)$, we have
    \begin{equation}\label{equ_thm_ray_radon}
        \lim_{n\to \infty} \|\Radon f-\RayRadonN f\|_{L^2(\sinodom)}=0.
        \tag{$\textrm{conv}^\mathrm{rd}$}
    \end{equation}
    If the sequence $(\delta^n)_{n\in \NN}$ satisfies $\frac{\delta_s^n}{\delta_x^n}\overset{n\to \infty}{\to} 0$, then, for each $g\in L^2(\sinodom)$, we have
        \begin{equation}\label{equ_thm_ray_estimate_backprojeciton}
        \lim_{n\to \infty} \|\Radon^* g-\RayRadonN^* g\|_{L^2(\imgdom)}=0.
        \tag{$\textrm{conv}^\mathrm{rd*}$}
    \end{equation}
        If $\frac{\delta_x^n}{\delta_s^n} \leq c$ for all $n\in \NN$, then, for each $g\in L^2(\sinodom)$, we have
    \begin{equation}\label{equ_thm_pixel_backprojection}
        \lim_{n\to \infty}\|\Radon^* g - \PixelRadonN^* g\|_{L^2(\imgdom)} =0
        \tag{$\textrm{conv}^\mathrm{pd*}$}.
    \end{equation}
       If the sequence $(\delta^n)_{n\in \NN}$ satisfies $\frac{\delta_x^n}{\delta_s^n}\overset{n\to \infty}{\to} 0$, then, for each $f\in L^2(\imgdom)$, we have
        \begin{equation}\label{equ_thm_pixel_radon}
        \lim_{n\to \infty} \|\Radon f-\PixelRadonN f\|_{L^2(\sinodom)}=0.
        \tag{$\textrm{conv}^\mathrm{pd}$}
    \end{equation}
\end{theorem}

\begin{remark}
Note that both \eqref{equ_thm_ray_radon} and \eqref{equ_thm_pixel_backprojection} are applicable in the case $\delta_x^n\approx \delta_s^n$. Hence, using the $\mathrm{rd}$-$\mathrm{pd}^*$ approach for balanced resolutions is indeed justified in the sense that we have pointwise convergence of the operators (in $\mathrm{SOT}$).
In the unbalanced case $\frac{\delta_s^n}{\delta_x^n}\to 0$, also the $\mathrm{rd}$-$\mathrm{rd}^*$ approach is justified in the sense of $\mathrm{SOT}$. In the unbalanced case $\frac{\delta_x^n}{\delta_s^n}\to 0$, the $\mathrm{pd}-\mathrm{pd}^*$ approach is analogously justified.

Note that the convergence described in Theorem \ref{Thm_approximation_ray_driven} is not necessarily uniform, i.e., the speed of convergence could depend significantly on the specific functions $f$ and $g$ considered and might potentially get arbitrarily slow. In the proof of Theorem \ref{Thm_approximation_ray_driven} below, we actually get concrete estimates on convergence rates for smooth functions $f$ and $g$; see \eqref{equ_proof_anglewise_convergence}, \eqref{equ_proof_estimate_ray-driven_backprojection}, \eqref{equ_proof_pixel-driven_backprojection_estimate} and \eqref{equ_proof_thm_pixel_driven_radon_estimate_angles}.

If we had convergence in the operator norm (which we do not, convergence in the operator norm is a stronger property than convergence in SOT), the convergence speed would be uniform for all $f$ and $g$ in the respective $L^2$ unit balls.
\end{remark}

For the sake of completeness, note that in the case $\frac{\delta_x^n}{\delta_s^n}\to 0$ and $\frac{\delta_\phi^n}{\delta_s^n}\to 0$, the pixel-driven Radon transform and backprojection converge in the operator norm as discussed in \cite{doi:10.1137/20M1326635}. The statement \eqref{equ_thm_pixel_radon} will be a consequence of this result.

We described above the convergence for the `full' angular setting $\phi \in [0,\pi[$. Due to technical limitations, in practice one often considers limited angle situations, i.e., one considers only $\phi \in \mathcal{A}$ for an interval $\mathcal{A}=[a,b[\subset [0,\pi[$; we set $\sinodom_{\mathcal{A}} = \mathcal{A}\times ]-1,1[$ the corresponding sinogram domain. The limited angle Radon transform  $\Radon_\mathcal{A}\colon L^2(\imgdom) \to L^2(\sinodom_\mathcal{A})$ is the restriction of the classical Radon transform to $\sinodom_\mathcal{A}$. 

Discretizing $\Radon_\mathcal{A}$, one can proceed analogously to the discretization described in Section \ref{section_Radon_discretization}, but we only discretize $\mathcal{A}$ instead of $[0,\pi[$.
Given angles $\phi_0<\dots<\phi_{N_\phi-1} \in \mathcal{A}$, we consider the corresponding angular pixels $\widetilde \Phi_q :=\left[\frac{\phi_{q-1}+\phi_q}{2},\frac{\phi_{q+1}+\phi_q}{2}\right[$ for $q\in \{1,\dots,N_s-2\}$, $ \widetilde \Phi_0=\left[a,\frac{\phi_{0}+\phi_1}{2}\right[$ and $\widetilde \Phi_{N_\phi-1}= \left[\frac{\phi_{N_\phi-2}+\phi_{N_\phi-1}}{2},b\right[$. 
We denote with $\GenRadon_\mathcal{A}$ and $\GenRadon_\mathcal{A}^*$ the definitions of the operators $\GenRadon$ and $\GenRadon^*$ as in \eqref{equ_def_convolutional_Radon_transform} and \eqref{equ_def_convolution_backprojection} when replacing $\Phi_q$ by $\widetilde \Phi_q$. These naturally map $L^2(\imgdom)\to L^2(\sinodom_\mathcal{A})$ and vice-versa, and can be considered as discretizations of $\Radon_\mathcal{A}$. 
The results of Theorem \ref{Thm_approximation_ray_driven} translate to the limited angle situation.
\begin{corollary}[SOT convergence for limited angles]
\label{Cor_limited_angle}
The convergence statements of Theorem \ref{Thm_approximation_ray_driven} remain valid if we replace:
$\Radon$ and $\Radon^*$ by $\Radon_\mathcal{A}$ and $\Radon_\mathcal{A}^*$; $\RayRadonN$ and $\RayRadonN^*$ by $\Radon_{\delta^n \mathcal{A}}^\mathrm{rd}$ and $\Radon_{\delta^n \mathcal{A}}^\mathrm{rd*}$ ; $\PixelRadonN$ and $\PixelRadonN^*$ by $\Radon_{\delta^n \mathcal{A}}^\mathrm{pd}$ and $\Radon_{\delta^n \mathcal{A}}^\mathrm{pd*}$; $L^2(\sinodom)$ by $L^2(\sinodom_\mathcal{A})$.
\end{corollary}

Note that these results do not inform about projections for specific angles $\phi$ converging. To also analyze the behavior for individual angles, we 
consider a finite angle set $\mathbb{F}:=\{\phi_0,\dots,\phi_{N_\phi-1}\}\subset [0,\pi[$ with $\phi_0< \dots< \phi_{N_\phi-1}$  and set $\sinodom_\mathbb{F} = \mathbb{F}\times ]-1,1[$. Correspondingly, we consider the space $L^2(\sinodom_\mathbb{F})$ equipped with the norm $\|g\|_{L^2(\sinodom_\mathbb{F})}^2 = \sumphi |\Phi_q| \int_{-1}^1 |g|^2(\phi_q,s) \dd s$.
Then, the sparse angle Radon transform $\Radon_\mathbb{F}\colon L^2(\imgdom) \to L^2(\sinodom_\mathbb{F})$ is defined as in \eqref{equ_def_radon_transform} but only for $\phi \in \mathbb{F}$.
We denote with $\GenRadon_{\mathbb{F}}$ and $\GenRadon_{\mathbb{F}}^*$ the definitions of $\GenRadon$ and $\GenRadon^*$ in \eqref{equ_def_convolutional_Radon_transform} and \eqref{equ_def_convolution_backprojection} when we replace $v_{qp}=\chi_{\Phi_q\times S_p}$ by $v_{qp}=\chi_{ \{\phi_q\}\times S_p}$ and $\int_{\Phi_q\times S_p} g (\phi,s) \dd {(\phi,s)}$ with $|\Phi_q| \int_{S_p} g (\phi_q,s)\dd s$.

And indeed, also each projection (for the individual angles in $\mathbb{F}$) converges in $L^2(]-1,1[)$ as we discuss next.
\begin{corollary}[SOT convergence for sparse angles]
\label{thm_convergence_sparse_angle}
The convergence statements of Theorem \ref{Thm_approximation_ray_driven} remain valid if we replace:
$\Radon$ and $\Radon^*$ by $\Radon_{\mathbb{F}}$ and $\Radon_{\mathbb{F}}^*$; $\RayRadonN$ and $\RayRadonN^*$ by $\Radon_{\delta^n {\mathbb{F}}}^\mathrm{rd}$ and $\Radon_{\delta^n {\mathbb{F}}}^\mathrm{rd*}$ ; $\PixelRadonN$ and $\PixelRadonN^*$ by $\Radon_{\delta^n {\mathbb{F}}}^\mathrm{pd}$ and $\Radon_{\delta^n {\mathbb{F}}}^\mathrm{pd*}$; $L^2(\sinodom)$ by $L^2(\sinodom_{\mathbb{F}})$.
\end{corollary}

\subsection{Proofs and technical details}
\label{section_proofs}

In order to prove Theorem \ref{Thm_approximation_ray_driven}, we need to discuss  some additional properties of the weight functions yielding exact approximations in certain situations.

\begin{lemma}[Exact weights] \label{Lemma_weight_length}
Let $f\in L^2(\imgdom)$, and given $\delta$, let $f_\delta =\sumx f_{ij}u_{ij}\in U_\delta$ with $f_{ij}=\frac{1}{\delta_x^2}\int_{X_{ij}}f \dd x$ (the projection of $f$ onto $U_\delta$). Let $(\phi,s)\in \sinodom$ and let $\hat q\in [N_\phi]$ and $\hat p\in [N_s]$ be such that $(\phi,s)\in \Phi_{\hat q}\times S_{\hat p}$. Then, we have
\begin{equation}
\label{equ_lemma_exact_on_U}
[\RayRadon f](\phi,s) = [\RayRadon f_\delta] (\phi_{\hat q},s_{\hat p})= [\Radon f_\delta] (\phi_{\hat q},s_{\hat p}).  
\tag{$\mathrm{exact}^\mathrm{rd}$}
\end{equation}
(Note that here $f_\delta \in U_\delta$ is understood as defined $\mathcal{H}^1$ almost everywhere (see Remark \ref{remark_using_hausdorff_ae}); thus, the pointwise evaluation $[\Radon f_\delta] (\phi_{\hat q},s_{\hat p})$ is well-defined.)

For fixed $\hat i,\hat j\in [N_x]$, the set $P_{\hat i,\hat j,\hat q}:=\{p\in [N_s] \ \big|\ \PixelWeight(x_{\hat i \hat j}\cdot \vartheta_{\hat q} -s_p)\neq 0\}$ satisfies
\begin{equation} \label{equ_lemma_relevant_p_pd}
P_{\hat i,\hat j,\hat q}  \begin{cases}
=\{\hat p\} & \text{if }x_{\hat i \hat j}\cdot\vartheta_{\hat q}=s_{\hat p}  \text{ for some }\hat p\in [N_s],
\\
=\{\hat p,\hat p+1\} \qquad &\text{if }x_{\hat i \hat j}\cdot\vartheta_{\hat q}\in ]s_{\hat p},s_{\hat p+1}[ \text{ for some }\hat p\in [N_s],
\\
\in \{\{0\},\{N_s-1\},\emptyset\} &  \text{else ( if $x_{\hat i \hat j}\cdot\vartheta_{\hat q} \not \in [s_{0},s_{N_s-1}]$)}.
\end{cases}
\end{equation}
Moreover, we have
\begin{equation}
\label{equ_lemma_pixel_sum_s}
    \sums \PixelWeight(x_{\hat i \hat j}\cdot \vartheta_{\hat q} -s_p) \quad  \begin{cases} =\frac{1}{\delta_s} \qquad & \text{if } x_{\hat i \hat j}\cdot \vartheta_{\hat q}\in [s_0,s_{N_s-1}],
    \\
    \leq \frac{1}{\delta_s} & \text{else}.
\end{cases}    
    \tag{$\mathrm{intpol}^{\mathrm {pd}}$}
\end{equation}

\end{lemma}

In other words, the ray-driven Radon transform is exact on $U_\delta$ for evaluation on sinogram pixel centers, and coincides with $\RayRadon f$ when $f_\delta$ is the projection of $f$ onto $U_\delta$. In contrast, the sum of pixel-driven weights equals $\frac{1}{\delta_s}$ while inside the detector range, which will result in an interpolation effect in the backprojection.

\begin{proof}[\textbf{Proof of Lemma \ref{Lemma_weight_length}}]
\underline{\eqref{equ_lemma_exact_on_U}:}  We consider $f_\delta=\sumx f_{ij} u_{ij}\in U_\delta$ with the coefficients $f_{ij}= \frac{1}{\delta_x^2}\int_{X_{ij}}f(x) \dd x$. 
We have $[\RayRadon f] (\phi,s)=[\RayRadon f] (\phi_{\hat q},s_{\hat p})$ for $(\phi,s)\in \Phi_{\hat q}\times S_{\hat p}$ (these  are constant functions on the sinogram pixels). Since $\int_{X_{ij}}f \dd x = \int_{X_{ij}}f_\delta \dd x$, we have  $[\RayRadon f] (\phi,s)=[\RayRadon f_\delta] (\phi_{\hat q},s_{\hat p})$ per definition in \eqref{equ_def_convolutional_Radon_transform}.

Moreover, we calculate
\begin{align*}
[\Radon f_\delta](\phi_{\hat q},s_{\hat p}) & \overset{\text{lin}}{=} \sumx f_{ij} [\Radon u_{ij}](\phi_{\hat q},s_{\hat p})
\\ 
&\underset{\text{def}}{\overset{\text{per}}{=}} \sumx f_{ij} \left ( \int_{\RR^2} \chi_{X_{ij}} \dd {\mathcal{H}^1 \mres L_{\phi_{\hat q},s_{\hat p}}}- \frac 12 \int_{\RR^2} \chi_{\partial X_{ij}} \dd {\mathcal{H}^1 \mres {L_{\phi_{\hat q},s_{\hat p}}}} \right)
\\
&
=\sumx f_{ij}\left ( \mathcal{H}^1(L_{\phi_{\hat q},s_{\hat p}} \cap X_{ij})- \frac{1}{2}\mathcal{H}^1(L_{\phi_{\hat q},s_{\hat p}} \cap \partial X_{ij}) \right)
\\
& \overset{\eqref{equ_lemma_intersectionlength}}{=} \sumx f_{ij} \delta_x^2\RayWeight (\phi_{\hat q},x_{ij}\cdot\vartheta_{\hat q} -s_{\hat p})
 \underset{}{\overset{\text{\eqref{equ_def_convolutional_Radon_transform}}}{=}} [\RayRadon f_\delta](\phi_{\hat q},s_{\hat p}) .
\end{align*}

\underline{ \eqref{equ_lemma_pixel_sum_s}:} Note that $ \PixelWeight(t)\neq 0$ iff  $t\in{]-\delta_s,\delta_s[}$.
If $x_{\hat i \hat j}\cdot\vartheta_{\hat q}=s_{\hat p}$ for some $\hat p \in [N_s]$, then $\PixelWeight(x_{\hat i \hat j}\cdot \vartheta_{\hat q}-s_{\hat p})=\PixelWeight(0)=\frac{1}{\delta_s}$ and 
\begin{equation}
|x_{\hat i \hat j}\cdot\vartheta_{\hat q}-s_{p}|=|s_{\hat p}-s_p| = |\hat p- p|\delta_s\geq \delta_s
\end{equation}
 for $p\neq \hat p$ and therefore $\PixelWeight(x_{\hat i \hat j}\cdot \vartheta_{\hat q}-s_p)=0$, implying \eqref{equ_lemma_relevant_p_pd} and \eqref{equ_lemma_pixel_sum_s}.

If $x_{\hat i \hat j}\cdot\vartheta_{\hat q}\in [s_0,s_{N_s-1}]$, but $x_{\hat i \hat j}\cdot\vartheta_{\hat q}\neq s_p$ for all $p\in [N_s]$, then there is a $\hat p\in [N_s-1]$ with $x_{\hat i \hat j}\cdot\vartheta_{\hat q}\in ]s_{\hat p},s_{\hat p+1}[$. Recall $s_p=s_0+p \delta_s$, and consequently 
\begin{equation} \label{equ_proof_which_p_are_relevant_pd}
|x_{\hat i \hat j}\cdot\vartheta_{\hat q}-s_p|= \min_{p^* \in \{\hat p,\hat p +1\}} |p-p^*|\delta_s+|x_{\hat i \hat j}\cdot\vartheta_{\hat q}-s_{p^*}|\geq \delta_s
\end{equation}
for $p\not \in \{\hat p,\hat p+1\}$ and thus $\PixelWeight(\phi_q,x_{\hat i \hat j}\cdot \vartheta_{\hat q}-s_p)=0$. 
We set $t = x_{\hat i \hat j}\cdot \vartheta_{\hat q}-s_{\hat p+1} \in ]-\delta_s,0[$, and $t+\delta_s = x_{\hat i \hat j}\cdot \vartheta_{\hat q}-s_{\hat p+1} +\delta_s= x_{\hat i \hat j}\cdot \vartheta_{\hat q}-s_{\hat p}$. 
Then, 
\begin{multline}
 \delta_s^2 \sums \PixelWeight(x_{\hat i \hat j}\cdot \vartheta_{\hat q} -s_p)=(\PixelWeight(x_{\hat i \hat j}\cdot \vartheta_{\hat q} -s_{\hat p})+\PixelWeight(x_{\hat i \hat j}\cdot \vartheta_{\hat q} -s_{\hat p+1}))\delta_s^2
\\
 = (\PixelWeight(t)+\PixelWeight(t+\delta_s))\delta_s^2 \underset{\text{}}{\overset{\eqref{equ_def_pixel_weight}}{=}} \delta_s +t + \delta_s - \delta_s-t = \delta_s,
\end{multline} implying \eqref{equ_lemma_pixel_sum_s}.

If $x_{\hat i \hat j}\cdot\vartheta_{\hat q} \not \in [s_0,s_{N_s-1}]$, there is at most one non-zero summand in \eqref{equ_lemma_pixel_sum_s} (via analogous consideration to \eqref{equ_proof_which_p_are_relevant_pd}), which is also bounded by $\frac{1}{\delta_s}$, implying the claim.
\end{proof}

In order to obtain suitable estimates, we require knowledge on the behavior of the sum of weights over all spatial or detector pixels.

\begin{lemma}[Sums of weights]
\label{Lemma_sums_weights}
    Given $\delta$, $\hat i,\hat j\in [N_x]$, $\hat q\in [N_\phi]$ and $\hat p\in [N_s]$, the following hold:
     \begin{equation} \label{equ_lemma_ray_sum_x}
        \sumx \RayWeight(\phi_{\hat q},x_{ij}\cdot \vartheta_{\hat q}-s_{\hat p}) \leq \frac{\sqrt{8}}{\delta_x^2}
        \tag{$\sum^\mathrm{rd}_{ij}$}.
    \end{equation}
     \begin{equation}\label{equ_lemma_ray_sum_p}
        \sums \RayWeight(\phi_{\hat q}, x_{\hat i \hat j}\cdot \vartheta_{\hat q} - s_p) \in \begin{cases}  \frac{1}{\delta_s}+   [-{\frac{\sqrt{8}}{\delta_x}},{\frac{\sqrt{8}}{\delta_x}}] 
       & \text{if }|x_{\hat i \hat j}\cdot\vartheta_{\hat q} |\leq 1 - \frac{\delta_x}{\sqrt{2}},
        \\
        [0,\frac{1}{\delta_s}+{\frac{\sqrt{8}}{\delta_x}}] & \text{otherwise}.
        \end{cases}
       \tag{$\sum^\mathrm{rd}_{p}$}  
    \end{equation}
     \begin{equation}\label{equ_lemma_pixel_sum}
         \sumx \PixelWeight(x_{ij}\cdot \vartheta_{\hat q}-s_{\hat p}) \leq \left\lceil\frac{\delta_s}{\delta_x}\right\rceil \frac{4\sqrt{2}   }{\delta_x\delta_s}
        \tag{$\sum^\mathrm{pd}_{ij}$},
    \end{equation}
    where $\lceil t\rceil := \min \{n\in \NN \ |\ t\leq n\}$.
\end{lemma}

\begin{remark}
Equation \eqref{equ_lemma_pixel_sum_s} would correspond to $(\sum_{p}^{\mathrm {pd}})$ in the notational logic of Lemma \ref{Lemma_sums_weights}. Note the difference in behavior between the \eqref{equ_lemma_pixel_sum_s} and \eqref{equ_lemma_ray_sum_p}; while the former attains a precise value $\frac{1}{\delta_s}$ that is meaningful, the latter only attains $\frac{1}{\delta_s}$ approximately with some inaccuracy depending on $\delta_x$. As will be seen in the proof of Theorem \ref{Thm_approximation_ray_driven}, this difference is chiefly responsible for requiring the stronger assumption $\frac{\delta_s^n}{\delta_x^n}\to 0$ in the ray-driven backprojection case (compared to the pixel-driven backprojection). A curious side effect of this difference is that the zero order data consistency condition
$\int_{-1}^1 [\Radon f] (\phi_1,s)\dd s= \int_{-1}^1  [\Radon f] (\phi_2,s)\dd s$ for any angles $\phi_1,\phi_2 \in [0,\pi[$ and $f\in \mathcal{C}^\infty_c(\imgdom)$ (see \cite{WOS:A1960WE64800004}) is also satisfied by $\PixelRadon f$ (evaluating these integrals leads precisely to the sums \eqref{equ_lemma_pixel_sum_s}), but not necessarily by $\RayRadon f$. 
\end{remark}

\begin{proof}[\textbf{Proof of Lemma \ref{Lemma_sums_weights}}]
\underline{\eqref{equ_lemma_ray_sum_x}:}   According to Lemma \ref{Lemma_weight_length}'s \eqref{equ_lemma_exact_on_U} for the function $f_\delta=\sumx 1 u_{ij}\in U_\delta$ (constantly one), we see that 
   \begin{equation}
    \delta_x^2\sumx  \RayWeight(\phi_{\hat q}, x_{ij}\cdot \vartheta_{\hat q}-s_{\hat p}) \underset{\text{def}}{\overset{\text{per}}{=}} [\RayRadon f_\delta](\phi_{\hat q},s_{ \hat p}) \overset{\eqref{equ_lemma_exact_on_U}}{=} [\Radon f_\delta](\phi_{\hat q},s_{ \hat p})  \leq \sqrt{8}, \notag
    \end{equation}
    where the last estimate is simply the maximal length ($\sqrt{8}$) of the ray in $[-1,1]^2$ times the maximal value of $f_\delta$ (being $1$). 

\underline{\eqref{equ_lemma_ray_sum_p}}: 
We note that the function $G(t) := \RayWeight(\phi_{\hat q},t)$ is monotone for $t\geq 0$ and $t\leq 0$, respectively (see Figure \ref{Fig_graphs_ray_driven}). Moreover, $G(t)\in [0,\frac{\sqrt{2}}{\delta_x}]$ for all $t\in \RR$, $\supp{G}\subset [-\frac{\delta_x}{\sqrt{2}},\frac{\delta_x}{\sqrt{2}}]$ and $\int_{\RR} G(t) \dd t = 1$ (using \eqref{equ_lemma_intersectionlength} and Fubini's theorem).
%
%
%

 For a function $G$ consisting of two monotone parts, a Riemann sum with step size $\delta_s$  can approximate the integral of $G$  up to $2 \delta_s \max_t G(t) $,  and therefore
\begin{equation} \label{equ_proof_riemann_sum}
    \Big|\delta_s \sum_{k=-\infty}^\infty \underbrace{\RayWeight(\phi_{\hat q},t_0+k \delta_s)}_{=G(t_0+k\delta_s)}\ - \underbrace{1}_{= \int_{\RR}  G(t) \dd t } \Big|\leq  \sqrt{8}\frac{\delta_s}{\delta_x} 
\end{equation}
for any $t_0\in \RR$. 
Setting $t_0 =s_0-x_{\hat i \hat j}\cdot \vartheta_{\hat q}$, we note $t_0+k \delta_s =s_{k}-x_{\hat i \hat j}\cdot \vartheta_{\hat q} $ for $k \in [N_s]$. If $|x_{\hat i \hat j}\cdot\vartheta_{\hat q}|<1-\frac{\delta_x}{\sqrt{2}}$ and $k\not \in [N_s]$, we have 
\begin{equation}
|s_0+k\delta_s-x_{\hat i \hat j}\cdot \vartheta_q| \geq |s_0+k \delta_s|-|x_{\hat i \hat j}\cdot \vartheta_q| \geq \left(1+\frac{\delta_s}{2}\right)-\left(1-\frac{\delta_x}{\sqrt{2}}\right)> \frac{\delta_x }{\sqrt{2}},
\end{equation} implying $\RayWeight(\phi_{\hat q},t_0+k\delta_s)=0$, i.e., all summands in \eqref{equ_proof_riemann_sum} for $k\not \in [N_s]$ vanish. Thus, \eqref{equ_proof_riemann_sum} but only with the summands for $k\in [N_s]$ and $t_0+k \delta_s =s_{k}-x_{ij}\cdot \vartheta_q $ yields \eqref{equ_lemma_ray_sum_p}.

We achieve the estimate \eqref{equ_lemma_ray_sum_p} if $|x_{\hat i \hat j}\cdot\vartheta_{\hat q}|\geq 1-\frac{\delta_x}{\sqrt{2}}$ by reformulation of \eqref{equ_proof_riemann_sum} according to 
\begin{equation}
\sums \RayWeight(\phi_{\hat q}, x_{\hat i \hat j}\cdot \vartheta_{\hat q} - s_p) \leq \sum_{k=-\infty}^\infty \RayWeight(\phi_{\hat q},t_0+k \delta_s) \overset{\eqref{equ_proof_riemann_sum}}{\leq} \frac{1}{\delta_s}+  \frac{\sqrt{8}}{\delta_x},
\end{equation}
where we used that all summands are non-negative.

\underline{\eqref{equ_lemma_pixel_sum}:} We wish to count the set $M_{\hat q, \hat p}:=\{(i,j)\in [N_x]^2 \ \big |\ |x_{ij}\cdot \vartheta_{\hat q}-s_{\hat p}|< \delta_s\}$, as those are the  pixels with non-zero contributions $\PixelWeight(x_{ij}\cdot \vartheta_{\hat q}-s_{\hat p})$ to \eqref{equ_lemma_pixel_sum} (as $\supp \PixelWeight \subset [-\delta_s,\delta_s]$).
We assume w.l.o.g. $\phi_{\hat q}\in \left[\frac{\pi}{4},\frac{3\pi}{4}\right]$ (and thus $|\sin(\phi_{\hat q})|\geq \frac {1}{\sqrt{2}}$). Fixing $\hat i$, the inequality 
\begin{equation}
 \delta_{s} > |s_{\hat p} -x_{\hat i j}\cdot \vartheta_{\hat q}|= |s_{\hat p}-x_{\hat i0}\cdot \vartheta_{\hat q}-j \delta_x \sin(\phi_{\hat q})|
 \end{equation}
  has at most $2\sqrt{2} \lceil\frac{\delta_s}{\delta_x}\rceil$ solutions for $j$ (there may be one even if $\delta_s \ll \delta_x$). Hence, summing up for all $\hat i\in [N_x]$, we have  $2\sqrt{2}N_x\lceil\frac{\delta_s}{\delta_x}\rceil$ relevant pixels (and $N_x=\frac{2}{\delta_x}$). The sum \eqref{equ_lemma_pixel_sum} can thus be estimated by the number of non-zero summands ($|M_{\hat q, \hat p}|\leq \frac{4\sqrt{2}}{\delta_x} \left\lceil\frac{\delta_s}{\delta_x}\right\rceil$) times the maximum of $\PixelWeight$ ($=\frac{1}{\delta_s}$), yielding \eqref{equ_lemma_pixel_sum}.
\end{proof}

Thanks to these estimates, we can show next that $\RayRadon$ and $\PixelRadon$ have bounded operator norms for reasonable choices of $\delta$. This is certainly a necessary condition to achieve convergence in the strong operator topology in Theorem \ref{Thm_approximation_ray_driven} (due to the uniform boundedness principle). On the other hand, uniform boundedness will be a crucial tool in proving Theorem \ref{Thm_approximation_ray_driven}.

\begin{lemma}[Uniformly bounded discretization] 
\label{Lemma_bounded_operators}
    Let $c>0$ be a constant. Then,
    \begin{equation} \label{equ_lemma_bounded_ray}
    \sup \left \{\|\RayRadon\| \ \Big | \ \delta=(\delta_x,\delta_\phi,\delta_s)\in (\RR^+)^3 \ \colon \ \frac{\delta_s}{\delta_x}\leq c \right \}   \leq \sqrt{\sqrt{8}\pi(1 +\sqrt{8}c)} <\infty
    \tag{$\mathrm{BD}^\mathrm{rd}$},
    \end{equation}
    \begin{equation}
        \label{equ_lemma_bounded_pixel}
       \sup \left \{\|\PixelRadon\| \ \Big | \ \delta=(\delta_x,\delta_\phi,\delta_s)\in (\RR^+)^3\  \colon \ \frac{\delta_x}{\delta_s}\leq c\right\} \leq \sqrt{4\sqrt{2}\pi (c+1)} <\infty
        \tag{$\mathrm{BD}^\mathrm{pd}$},
    \end{equation}
    where $\|\cdot \|$ refers to the operator norm for operators from $L^2(\imgdom)$ to $L^2(\sinodom)$.
\end{lemma}

\begin{proof}[\textbf{Proof of Lemma \ref{Lemma_bounded_operators}}]
    \underline{\eqref{equ_lemma_bounded_ray}:} 
 For each $q\in[N_\phi]$, $p\in[N_s]$, we define the measure  
\begin{equation} \label{equ_lemma_bounded_mu_1}
 \mu_{qp}:=\frac{1}{\sqrt{8}} \sumx \RayWeight(\phi_q, x_{ij}\cdot \vartheta_q-s_p) \mathcal{L}^2 \mres X_{ij}
 \end{equation} with $\mathcal{L}^2\mres X_{ij}$ the two-dimensional Lebesgue measure restricted to $X_{ij}$. Note that $\mu_{qp}$ is a sub-probability measure thanks to \eqref{equ_lemma_ray_sum_x}.
    Given $f\in L^2(\imgdom)$, and using Jensen's inequality for $\mu_{qp}$, we have
    \begin{align*}
       & \|\RayRadon f\|_{L^2(\sinodom)}^2 
        \overset{\eqref{equ_def_convolutional_Radon_transform}}{=}
         \sumphi \sums \delta_s |\Phi_q|\left|\sumx \RayWeight(\phi_q,x_{ij}\cdot \vartheta_q-s_p) \int_{X_{ij}} f(x) \dd x \right|^2 \notag
         \\
                 &\overset{\eqref{equ_lemma_bounded_mu_1}}{=} 8
         \sumphi \sums \delta_s |\Phi_q|\left| \int_{\imgdom} f(x) \dd{ \mu_{qp}(x)} \right|^2 \notag        
        \overset{\text{Jen}}{\leq}  8\sumphi \sums \delta_s |\Phi_q| \int_{\imgdom} |f(x)|^2 \dd {\mu_{qp}(x)}
        \\
  &\overset{\eqref{equ_lemma_bounded_mu_1}}{=}      \sqrt{8} 
        \sumx \int_{X_{ij}} \left|f(x) \right|^2\dd x \sumphi |\Phi_q| \sums \delta_s   \RayWeight(\phi_q, x_{ij}\cdot \vartheta_q-s_p)  
        \\
         &\overset{\eqref{equ_lemma_ray_sum_p}}{\leq}
          \sqrt{8}\pi \left(1+\sqrt{8}\frac{\delta_s}{\delta_x}\right) \|f\|_{L^2(\imgdom)}^2, \notag
    \end{align*}
    where we used $\|f\|_{L^2(\imgdom)}^2= \sumx \int_{X_{ij}} \left|f(x) \right|^2\dd x$ and $\sumphi |\Phi_q|=\pi$.
    Consequently, $\|\RayRadon\|^2\leq \sqrt{8}\pi(1 +\sqrt{8}c)$ if $\frac{\delta_s}{\delta_x}\leq c$.

\underline{\eqref{equ_lemma_bounded_pixel}}:
For $i,j\in[N_x]$, we define the (sub-probability due to \eqref{equ_lemma_pixel_sum_s}) measure 
\begin{equation}\label{equ_lemma_bounded_mu_2}
\nu_{ij}:= \frac{1}{\pi}\sumphi\sums  \PixelWeight(x_{ij}\cdot \vartheta_q-s_p) \mathcal{L}^2\mres (\Phi_q\times S_p).
\end{equation}
Given $g\in L^2(\sinodom)$, we use Jensen's inequality to get
\begin{align}
 \label{equ_proof_pixel_thm_jensens}
 &  \|\PixelRadon^* g\|_{L^2(\imgdom)}^2 \overset{\eqref{equ_def_convolution_backprojection}}{\leq} \delta_x^2 \sumx  \left| \sumphi\sums  \PixelWeight(x_{ij}\cdot \vartheta_q-s_p) \int_{\Phi_q\times S_p} g(\phi,s) \dd{(\phi,s)} \right| ^2  \notag
    \\
& \qquad \overset{\eqref{equ_lemma_bounded_mu_2}}{=} \delta_x^2 \pi^2 \sumx  \left| \int_{\sinodom} g(\phi,s) \dd{\nu_{ij} (\phi,s)} \right| ^2  
\overset{\text{Jen}}{\leq}
\delta_x^2 \pi^2 \sumx  \int_{\sinodom} |g(\phi,s)|^2 \dd{\nu_{ij}(\phi,s)}\notag
    \\    
&\qquad    \overset{\eqref{equ_lemma_bounded_mu_2}}{=}   \delta_x^2 \pi\sumx \left(\sumphi \sums \PixelWeight(x_{ij}\cdot \vartheta_q -s_p) \int_{\Phi_q\times S_p} |g(\phi,s)|^2 \dd{(\phi,s)}  \right). 
\end{align}
Pulling the sum $\sumx$ into the other summands, using \eqref{equ_lemma_pixel_sum} and $\|g\|^2_{L^2(\sinodom)}=\sumphi\sums \int_{\Phi_q\times S_p} |g|^2 \dd{(\phi,s)}$, we see  $\|\PixelRadon^* g\|_{L^2(\imgdom)}^2\leq 4\sqrt{2} \pi \frac{\delta_x}{\delta_s} \left \lceil \frac{\delta_s}{\delta_x} \right\rceil  \|g\|_{L^2(\sinodom)}^2$. 
If $\frac{\delta_x}{\delta_s}\leq c$, then $\left \lceil \frac{\delta_s}{\delta_x}\right \rceil \leq (c+1) \frac{\delta_s}{\delta_x}$ (see footnote\footnote{Given $z>0$, we have $\lceil z\rceil \leq z+1=  \frac{(z+1)}{z} z$. The ratio $\frac{(z+1)}{z}$ is monotonically decreasing in $z$. When $\frac{1}{z}\leq c$ we have $z\geq 1/c$. This yields $\lceil z \rceil \leq  \frac{z+1}{z} z \leq (c+1) z$ with $(c+1)=\frac{z+1}{z}$ for $z=\frac{1}{c}$ .
}), implying $\|\PixelRadon\|^2 = \|\PixelRadon^*\|^2 \leq  4\sqrt{2}\pi (c+1)$ for such $\delta$.
\end{proof}

\begin{proof}[\textbf{Proof of Theorem \ref{Thm_approximation_ray_driven}}]
	The proofs of \eqref{equ_thm_ray_radon}, \eqref{equ_thm_ray_estimate_backprojeciton}, \eqref{equ_thm_pixel_backprojection} and \eqref{equ_thm_pixel_radon} will work as follows. First, we show convergence for smooth functions using Taylor's theorem and estimates from Lemmas \ref{Lemma_weight_length} and \ref{Lemma_sums_weights}. Once this is achieved, the convergence statements for general $L^2$ functions is obtained using a diagonal argument that exploits the boundedness described in Lemma \ref{Lemma_bounded_operators}.

    \underline{\eqref{equ_thm_ray_radon}:} Let $f\in \mathcal{C}^\infty_c(\imgdom)$ (infinitely differentiable and compactly supported). 
Let $\delta = (\delta_x,\delta_\phi,\delta_s)$ be some discretization parameter (we will specify $\delta$ later, for now it is generic).
    Let $(\phi,s)\in \sinodom$ and let $q\in [N_\phi]$ and $p\in [N_s]$ be such that $(\phi,s)\in \Phi_q\times S_p$. Using the triangle inequality for integrals, Taylor's theorem ($|f(y)-f(x)| \leq \| \nabla f\|_{L^\infty} \|y-x\|$), and the supremum estimate $\int_A f \dd x \leq |A| \sup_{x\in A}\{f(x)\}$, we have
    \begin{align} 
\label{equ_proof_ray_driven_local_est1}
        &|[\Radon f](\phi,s)-[\Radon f](\phi_q,s_p)| \underset{\eqref{equ_def_radon_transform}}{\overset{\text{Tri}}{\leq}} \int_{-1}^1 |f(s\vartheta_\phi+t \vartheta_\phi^\perp) - f(s_p\vartheta_q+t \vartheta_q^\perp)| \dd t\notag
        \\
        &\overset{\text{Tay}}{\leq}  \int_{-1}^1 \|\nabla f\|_{L^\infty}  \|(s\vartheta_\phi+t \vartheta_\phi^\perp) - (s_p\vartheta_q+t \vartheta_q^\perp)  \| \dd t
        \\
        &\overset{\text{Sup}}{\leq} 2\|\nabla f\|_{L^\infty} \max_{t\in [-1,1]}\|(s\vartheta_\phi+t \vartheta_\phi^\perp) - (s_p\vartheta_q+t \vartheta_q^\perp)  \| \leq  4 \|\nabla f\|_{L^\infty}(\delta_s+\delta_\phi),\notag
        \end{align}
where we estimated the maximum term by $ \frac{\delta_s}{2}+ 2 \delta_\phi$ since $s\in S_p$ and $\phi \in \Phi_q$.   (Note that $f$ being smooth, $\Radon f$ is defined pointwise and not only almost everywhere.)
    
    We set $f_\delta=\sumx f_{ij} u_{ij}\in U_\delta$ with $f_{ij}=\frac{1}{\delta_x^2}\int_{X_{ij}} f(\tilde x) \dd {\tilde x}$ (the function $f_\delta$ is again understood  $\mathcal{H}^1$ almost everywhere).
         Using Taylor's theorem, noting that $|x-\tilde x|\leq \sqrt{2} \delta_x$ if $x,\tilde x \in X_{ij}$ and $\frac{1}{\delta_x^2}\int_{X_{ij}} \dd {\tilde x} =1$, we have 
\begin{align}\label{equ_proof_taylor_forward}
     |f(x)-f_\delta(x)| &= \left| \frac{1}{\delta_x^2}\int_{X_{ij}} f(x)- f(\tilde x) \dd {\tilde x} \right| \overset{\text{Tay}}{\leq} \frac{1}{\delta_x^2}\int_{X_{ij}} \|\nabla f\|_{L^\infty} \|x-\tilde x\| \dd {\tilde x} \notag
     \\ 
     &
     \leq \sqrt{2} \delta_x\|\nabla f\|_{L^{\infty}} \quad \text{for $\mathcal{H}^1$ almost all }x\in \imgdom
     \end{align}     
 (the only exceptions are corners of pixels $X_{ij}$). Thus, we have
    \begin{multline} \label{equ_proof_ray_driven_local_est2}
        \left|[\Radon f](\phi_q,s_p) - \left[\RayRadon f\right](\phi,s)\right| \overset{\eqref{equ_lemma_exact_on_U}}{=}\left|[\Radon (f-f_\delta)](\phi_q,s_p)\right| 
        \\
        \overset{\eqref{equ_proof_taylor_forward}}{\leq} 2 \sqrt{2} \delta_x
        \|\nabla f\|_{L^\infty} ,
    \end{multline}
    where we used $|\Radon f|(\phi,s) \leq 2 \|f\|_{L_{\mathcal{H}^1}^\infty}$ for all $(\phi,s)\in \sinodom$.
    Combining \eqref{equ_proof_ray_driven_local_est1} and \eqref{equ_proof_ray_driven_local_est2}, we see that 
	\begin{equation} \label{equ_proof_anglewise_convergence}
    \left | [\Radon f - \RayRadon f](\phi,s) \right|\leq 4 (\delta_x+\delta_\phi+\delta_s)\|\nabla f\|_{L^\infty} \qquad \text{for all }(\phi,s)\in \sinodom.
    \end{equation}	    
 Note that so far $\delta$ was generic, but from \eqref{equ_proof_anglewise_convergence} we conclude $\|\Radon f-\RayRadonN f\|_{L^2(\sinodom)} \to 0$ as $n\to \infty$ (and $\delta^n\to 0$).

    Let $f\in L^2(\imgdom)$ be not necessarily smooth or compactly supported, and let $\epsilon>0$. There is an $\tilde f\in \mathcal{C}^\infty_c(\imgdom)$ such that $\|f-\tilde f\|_{L^2(\imgdom)}\leq \epsilon$ (since $\mathcal{C}^\infty_c(\imgdom)$ is dense in $L^2(\imgdom)$).  There is an $N=N(\epsilon,\tilde f)\in \NN_0$ such that $\|\Radon \tilde f - \RayRadonN\tilde f\|_{L^2(\sinodom)}\leq \epsilon$ for all $n>N$ (as discussed in the previous paragraph). Then, for $n>N$, we have
   \begin{multline} \label{equ_proof_diagonal_argument}
        \|\Radon f-\RayRadonN f\|_{L^2(\sinodom)}\leq \|\Radon f-\Radon \tilde f\|_{L^2} +\|\Radon \tilde f - \RayRadonN \tilde f  \|_{L^2}+\|\RayRadonN \tilde f -\RayRadonN f \|_{L^2}  
        \\
        \leq
        (\|\Radon\|+\|\RayRadonN\|) \|f-\tilde f\|_{L^2(\imgdom)} + \|\Radon \tilde f - \RayRadonN \tilde f  \|_{L^2(\sinodom)} \leq  C \epsilon,
    \end{multline}
    where $C=\|\Radon\|+\sup_n\{\|\RayRadonN\|\}+1<\infty$ using \eqref{equ_lemma_bounded_ray} with $\frac{\delta_s^n}{\delta_x^n}\leq c$ (as assumed). Thus, for any $\epsilon>0$, we have  $\|\Radon f- \RayRadonN f\|_{L^2(\sinodom)}\leq C \epsilon$ for all  $n>M=M(\epsilon,f)=N(\epsilon,\widetilde f)$ for some $\widetilde f \in \mathcal{C}^\infty_c(\imgdom)$ with $\|f-\widetilde f\|\leq \epsilon$, implying \eqref{equ_thm_ray_radon}.


\underline{\eqref{equ_thm_ray_estimate_backprojeciton} \& \eqref{equ_thm_pixel_backprojection}:}
Let $g\in \mathcal{C}^\infty_c(\sinodom)$. Due to the compact support, we have $g(\phi,s)=0$ for all $\phi$ and $|s|>c$ with some constant $0<c<1$.

  Given $x \in \imgdom$, let $i,j\in [N_x]$ be such that $x\in X_{ij}$. We reformulate the definition of $\Radon^*$ and $\GenRadon^* $  (in \eqref{equ_def_backprojection} and \eqref{equ_def_convolution_backprojection}) to see
 \begin{align} \label{equ_proof_backproj_estimate}
        [\Radon^* g- \GenRadon^* g](x) \notag
        = \sumphi \int_{\Phi_q} &g(\phi,x\cdot \vartheta_\phi) \Big(1-\delta_s\sums \I(ij,q,p) \Big ) 
        \\
        &+ \sums \I(ij,q,p) \II(x,\phi,p) \dd \phi 
        \end{align}
\begin{equation}
 \text{with} \quad        \I(ij,q,p ):=
\omega(\phi_q,x_{ij}\cdot \vartheta_q-s_p), 
\quad
\II(x,\phi,p) :=\int_{S_p} g(\phi, x\cdot \vartheta_\phi) - g(\phi,s) \dd s. \notag
    \end{equation}
The approach for showing both  \eqref{equ_thm_ray_estimate_backprojeciton} and \eqref{equ_thm_pixel_backprojection} is quite similar when considering \eqref{equ_proof_backproj_estimate} with $\omega=\RayWeight$ or $\omega=\PixelWeight$. We want $\sums\I(ij,q,p) \approx \frac{1}{\delta_s}$, and when $\I(ij,q,p)\neq 0$, we estimate $\II(x,\phi,p)$ using Taylor's theorem. Thus, we will obtain pointwise convergence for this fixed smooth $g$.
The conclusion for general $g\in L^2(\sinodom)$ then follows via a diagonal argument analogous to \eqref{equ_proof_diagonal_argument}.


   \underline{\eqref{equ_thm_ray_estimate_backprojeciton}:} If $|x_{ij} \cdot\vartheta_q|<1-\frac{\delta_x}{\sqrt{2}} $  for fixed $q$, we have $\sums \I(ij,q,p) \in \frac{1}{\delta_s}+[-\frac{\sqrt{8}}{\delta_x},\frac{\sqrt{8}}{\delta_x}]$ according to \eqref{equ_lemma_ray_sum_p}. When $|x_{ij} \cdot\vartheta_q|\geq 1-\frac{\delta_x}{\sqrt{2}} $ on the other hand, we have $g(\phi,x\cdot\vartheta_\phi)=0$
since  $\|x-x_{ij}\|\leq\frac{\delta_x}{\sqrt{2}}$, $\|\vartheta_\phi-\vartheta_q\|\leq \delta_\phi$ and 
\begin{equation}
|x\cdot \vartheta_\phi|\geq |x_{ij}\cdot \vartheta_q|-\frac{1}{\sqrt{2}}\delta_x-\delta_\phi \geq 1 - \sqrt{2} \delta_x-\delta_\phi >c
\end{equation}
 for $\delta_x$ and $\delta_\phi$ sufficiently small (recall $\supp{s\mapsto g(\phi,s)}\subset [-c,c]$ for all $\phi$ by assumption).  In conclusion, the summand in the first row of \eqref{equ_proof_backproj_estimate} is bounded by $\sqrt{8}\frac{\delta_s}{\delta_x} \|g\|_{L^\infty}$ if $\delta$ is sufficiently small.
    
  In order to estimate $\II$, we estimate the difference in arguments of $g$ in $\II$ by
    \begin{multline}\label{equ_proof_thm_ray_intermediary}
    |x\cdot \vartheta_\phi-s|\leq \overbrace{\|x-x_{ij}\|}^{\leq \frac{1}{\sqrt{2}}\delta_x} + \overbrace{|x_{ij}\cdot\vartheta_\phi-x_{ij}\cdot\vartheta_q|}^{\leq \|x_{ij}\| \|\phi-\phi_q\| \leq \sqrt{2} \delta_\phi}+\overbrace{\leq |x_{ij}\cdot\vartheta_q-s_p|}^{\frac{1}{\sqrt{2}}\delta_x}+\overbrace{|s_p-s|}^{\leq \frac{1}{2}\delta_s}
    \\
    \leq \frac{1}{2}\delta_s+\sqrt{2}\delta_x+\sqrt{2}\delta_\phi\leq \frac{3}{2}(\delta_x+\delta_s+\delta_\phi)
    \end{multline}
    if $\phi\in \Phi_q$, $s\in S_p$, $x\in X_{ij}$  and $\I(ij,q,p)\neq 0$ (i.e., $| x_{ij}\cdot \vartheta_q-s_p|<\frac{\delta_x}{\sqrt{2}}$) and $\|x_{ij}\|\leq \sqrt{2}$.   We use Taylor's theorem to estimate that if $\I(ij,q,p)\neq 0$, then
\begin{equation}\label{equ_proof_thm_ray_secondterm}
|\II(x,\phi,p)| \overset{\text{Tri}}{\leq } \int_{S_p} |g(\phi,x\cdot \vartheta_\phi)- g(\phi,s)| \dd s  \overset{\text{Tay}}{\underset{\eqref{equ_proof_thm_ray_intermediary}}{\leq}} \frac{3}{2}\delta_s(\delta_x+\delta_s+\delta_\phi) \left \|\frac{\partial g}{\partial s} \right \|_{L^\infty}.
\end{equation}   
Moreover,  we have $\sums |\I(ij,q,p)|\leq \frac{1}{\delta_s}+{\frac{\sqrt{8}}{\delta_x}}$ in any case; see \eqref{equ_lemma_ray_sum_p}. So the second line of \eqref{equ_proof_backproj_estimate} can be bounded by $\frac{3}{2}(1+\sqrt{8}\frac{\delta_s}{\delta_x})(\delta_x+\delta_s+\delta_\phi)\|\frac{\partial g}{\partial s}\|_{L^\infty}$.
Therefore, according to \eqref{equ_proof_backproj_estimate}, we have
\begin{multline}
   \left| [\Radon^* g- \RayRadon^* g](x)\right|\\\leq \int_0^{\pi} \sqrt{8} \frac{\delta_s}{\delta_x} \|g\|_{L^\infty} + 
     \frac{3}{2} \left (1+\sqrt{8}\frac{\delta_s}{\delta_x} \right) (\delta_x+\delta_s+\delta_\phi) \left \|\frac{\partial g}{\partial s} \right \|_{L^\infty}  \dd \phi .
\end{multline}
         Note that the integrand does not depend on $\phi$, resulting in
    \begin{multline} \label{equ_proof_estimate_ray-driven_backprojection}
   \left| [\Radon^* g- \RayRadon^* g](x)\right| \\\leq \sqrt{8} \pi \frac{\delta_s}{\delta_x} \|g\|_{L^\infty}+ \frac{3}{2} \pi  \left (1+\sqrt{8}\frac{\delta_s}{\delta_x} \right) (\delta_x+\delta_s+\delta_\phi) \left \|\frac{\partial g}{\partial s} \right \|_{L^\infty}
    \end{multline}   
    if $\delta$ is sufficiently small.
When increasing $n\to \infty$ (and thus $\delta^n\to 0$), we therefore obtained the desired convergence $\|\Radon^* g- \RayRadonN^* g\|_{L^2(\sinodom)}\to 0$ for smooth functions $g$ (assuming $\frac{\delta_s^n}{\delta_x^n} \overset{n\to \infty}{\to} 0 $). The convergence for general $g\in L^2(\sinodom)$ follows using a diagonal argument analogous to \eqref{equ_proof_diagonal_argument} (with $\|\RayRadon^*\|$ bounded due to \eqref{equ_lemma_bounded_ray}).

\underline{\eqref{equ_thm_pixel_backprojection}:} For $|x_{ij}\cdot\vartheta_q|<1-\frac {\delta_s}{2}$ ($=|s_0|=|s_{N_s-1}|$), via \eqref{equ_lemma_pixel_sum_s} we have $\sums \I(ij,q,p)=\frac{1}{\delta_s}$. Again, if $|x_{ij}\cdot\vartheta_q|\geq 1-\frac {\delta_s}{2} $, the corresponding $g(\phi,x\cdot \vartheta_\phi)$ (with $x\in X_{ij}$ and $\phi \in \Phi_q$) equals  $0$ for $\delta$ sufficiently small since 
\begin{equation}
|x\cdot\vartheta_\phi|> |x_{ij}\cdot\vartheta_q| -\frac{\delta_x}{\sqrt{2}}-\delta_\phi> 1 - \frac{\delta_s}{2}-\frac{\delta_x}{\sqrt{2}}-\delta_\phi>c,
\end{equation}
where again $c$ is the upper bound on the $s$ variables in the support of $g$.
Hence, the summand in the first row of \eqref{equ_proof_backproj_estimate} is zero.

Again, $|\II (x,\phi,p)|\leq  \frac{3}{2}\delta_s(\delta_s+\delta_x+\delta_\phi)\left \|\frac{\partial g}{\partial s} \right \|_{L^\infty}$ if $\I(ij,q,p)\neq 0$ (with $x\in X_{ij}$ and $\phi \in \Phi_q$) using Taylor's theorem like in \eqref{equ_proof_thm_ray_secondterm}.  Moreover, $\sums |\I(ij,q,p)| \leq \frac{1}{\delta_s}$.

In conclusion, we have
\begin{align}
\label{equ_proof_pixel-driven_backprojection_estimate}
    | [\Radon^* g-\PixelRadon^* g](x) |
    &
    \leq  \int_0^\pi  (\delta_s+\delta_x+\delta_\phi)\frac{3}{2} \left \|\frac{\partial g}{\partial s} \right \|_{L^\infty}\dd \phi\notag
    \\ 
    &
    = \frac{3}{2}\pi (\delta_s+\delta_x+\delta_\phi) \left \|\frac{\partial g}{\partial s} \right \|_{L^\infty}
\end{align}
for all $x\in \imgdom$ if $\delta$ is sufficiently small, implying $\|\Radon^* g-\PixelRadonN^* g\|_{L^2(\imgdom)} \to 0$ as $n\to \infty$.
The proof for general $g\in L^2(\sinodom)$ follows again by a diagonal argument analogous to \eqref{equ_proof_diagonal_argument} (using \eqref{equ_lemma_bounded_pixel}).

\underline{\eqref{equ_thm_pixel_radon}}
Let $f\in \mathcal{C}^\infty_c(\imgdom)$. 
From \cite[Theorem 2.16 and its proof]{doi:10.1137/20M1326635} and \cite[Theorem III.2.30]{huber2022pixel}, it is known that
\begin{multline}
\|[\Radon f] (\phi_q,\cdot) - [\PixelRadon f](\phi_q,\cdot) \|_{L^2(]-1,1[)} 
\\ 
\leq c \left ( \sup_{|h|\leq \frac{3}{2}\delta_s}M_{[\Radon f](\phi_q,\cdot)}(h) + \sqrt{1+\frac{\delta_x}{\delta_s}}\frac{\delta_x}{\delta_s} \|f\|_{L^2}\right),
\end{multline}
where the constant $c$ does not depend on the specific $\phi_q$ or $f$ and the $L^2$ modulus of continuity is $M_g(h)^2:= \int_{-1}^1 | g(s+h)-g(s)|^2 \dd s$.
It is well-known that the Radon transform of a smooth function satisfies $\| \frac{\partial [\Radon f](\phi_q,\cdot)}{\partial s} \|_{L^\infty} \leq 2 \|\nabla f \|_{L^\infty}$, and using Taylor's theorem it is trivial to estimate $M_g(h)\leq \sqrt{2}h\|\frac{\partial g}{\partial s}\|_{L^\infty}$. Consequently, for smooth $f$, we have
\begin{multline}\label{equ_proof_thm_pixel_driven_radon_estimate_for_fixed_angle}
\|[\Radon f] (\phi_q,\cdot) - [\PixelRadon f](\phi_q,\cdot) \|_{L^2(]-1,1[)} 
\\
\leq c \left( 3  \sqrt{2} \delta_s\|\nabla f\|_{L^\infty}+ \sqrt{1+\frac{\delta_x}{\delta_s}}\frac{\delta_x}{\delta_s} \|f\|_{L^2}\right),
\end{multline}
where again $c$ does not depend on $\phi$ or $f$. Hence, for $\phi \in \Phi_q$ for some fixed $q$, we have
\begin{multline} \label{equ_proof_thm_pixel_driven_radon_estimate_angles}
\|[\Radon f] (\phi,\cdot) - [\PixelRadon f](\phi,\cdot) \|_{L^2(]-1,1[)}  
\\
\leq \|[\Radon f] (\phi,\cdot) - [\Radon f](\phi_q,\cdot) \|_{L^2(]-1,1[)} 
 + \|[\Radon f] (\phi_q,\cdot) - [\PixelRadon f](\phi_q,\cdot) \|_{L^2(]-1,1[)} 
\\
\leq 2 
\left\|\frac{\partial [\Radon f] }{\partial \phi} \right \|_{L^\infty} \delta_\phi + c \left( 3  \sqrt{2}\delta_s\|\nabla f\|_{L^\infty}+ \sqrt{1+\frac{\delta_x}{\delta_s}}\frac{\delta_x}{\delta_s} \|f\|_{L^2}\right),
\end{multline}
where we used Taylor's theorem for the first term and \eqref{equ_proof_thm_pixel_driven_radon_estimate_for_fixed_angle} for the second term.
Integrating \eqref{equ_proof_thm_pixel_driven_radon_estimate_angles} with respect to $\phi$ yields that 
\begin{equation}
\|\Radon f- \PixelRadon f\|_{L^2(\sinodom)} \leq k \left(\delta_\phi+\delta_s+\frac{\delta_x}{\delta_s}\right),
\end{equation}
where the constant $k=k(f)>0$ depends on $f$ but not on $\delta$ (when $\frac{\delta_x}{\delta_s}$ remains bounded). Hence, we achieved $\PixelRadonN f \to \Radon f$ in the $L^2$ norm as $\delta^n \to 0$ (assuming $\frac{\delta_x^n}{\delta_s^n} \to 0$) for smooth functions $f$. For general $L^2$ functions, again a diagonal argument as in \eqref{equ_proof_diagonal_argument} is possible, yielding the desired estimate.

\end{proof}

\begin{proof}[\textbf{Proof of Corollary \ref{Cor_limited_angle}}]
Given the angles $\phi_q \in \mathcal{A}$ and related $\widetilde \Phi_q$ for $q\in[N_\phi]$, there is a set of angles $\varphi_0,\dots,\varphi_N$ such that the set of corresponding angular pixels $\Phi_k$ for $k\in[N]$ contains all $\widetilde \Phi_q$ for $q\in [N_\phi]$, $\{\phi_q \ | \ q \in [N_\phi]\}\subset \{\varphi_k \ | \ k \in [N]\}$ and $\delta_\phi=\max_{q\in[N_\phi]} |\widetilde \Phi_q|$. Thus, $\GenRadon_{\mathcal{A}}f$ is the restriction of $\GenRadon f$ (created with these $\varphi_k$ angles), so convergence in the strong operator topology of $\GenRadon$ immediately implies convergence for $\GenRadon_{\mathcal{A}}$. Similarly, $\GenRadon_{\mathcal{A}}^* g=\GenRadon^* \tilde g$ if $\tilde g = g $ on $\sinodom_\mathcal{A}$ and zero otherwise. Naturally, convergence of $\GenRadon^* $ in the strong operator topology then implies the same convergence for $\GenRadon{\mathcal{A}}^* $.
\end{proof}

\begin{proof}[\textbf{Proof of Corollary \ref{thm_convergence_sparse_angle}}]
Since $L^2(\sinodom_\mathbb{F})$ is not a subspace of $L^2(\sinodom)$, we cannot replicate the proof for Corollary \ref{Cor_limited_angle}.
Obviously, $\|\GenRadon_{\mathbb{F}} f\|_{L^2(\sinodom_\mathbb{F})}=\|\GenRadon f\|_{L^2(\sinodom)}$ (for each $f\in L^2(\imgdom)$) and therefore Lemma \ref{Lemma_bounded_operators} remains true for $\Radon_{\delta \mathbb{F}}^\mathrm{rd}$ and $\Radon_{\delta \mathbb{F}}^\mathrm{pd}$. One can follow the proof of Theorem \ref{Thm_approximation_ray_driven} word for word when replacing integrals with regard to $\phi$ with sums over $\mathbb{F}$. Important details are that \eqref{equ_proof_anglewise_convergence} holds for all $\phi \in [0,\pi[$ (and not only almost all) thus it also holds for all $\phi\in \mathbb{F}$, and the estimates for the integrands inside \eqref{equ_proof_backproj_estimate} we performed hold for all angles $\phi \in [0,\pi[$, thus in particular also for $\phi \in \mathbb{F}$.
\end{proof}

\section{Numerical Aspects}
\label{section_numerics}

In this section, we describe the implementation of convolutional discretizations and perform numerical experiments to complement the presented theoretical results.

\subsection{Implementation}
Following the formulation in \eqref{equ_def_convolutional_Radon_discrete}, Algorithm \ref{algo:Pixel-Driven-Radon} describes the implementation of the discrete convolutional forward projection $[\GenRadon f](\phi_q,s_p) \approx (A \overline f)[qp]$ for a phantom $f$ and the corresponding vector $\overline f$ (as an $N_x\times N_x$ array $\overline f$ containing the coefficients $\overline f [ij]=\frac{1}{\delta_x^2}\int_{X_{ij}} f \dd x$) and one specific pair of indices $(q,p) \in [N_\phi]\times [N_s]$. 
Analogously, Algorithm \ref{algo:Pixel-Driven-backprojection} describes the implementation of the discrete convolutional backprojection $[\GenRadon^* g] (x_{ij}) \approx (B \overline g)[ij]$ (see \eqref{equ_def_convolutional_backproj_discrete}) for an $N_\phi\times N_s$ sinogram array $\overline g$ with $\overline g[qp]= \frac{1}{\delta_s|\Phi_q|} \int_{\Phi_q\times S_p} g \dd{(\phi,s)}$ and one specific pixel center $x_{ij}$ (with indices $i,j\in [N_x]$).

Note that the entire discrete Radon transform (or backprojection) can be calculated by executing Algorithm \ref{algo:Pixel-Driven-Radon} (or Algorithm \ref{algo:Pixel-Driven-backprojection}) for each sinogram pixel $(q,p)$ (or space pixel $i,j$) individually. In particular, this is highly parallelizable.

\begin{algorithm}
    \caption{Convolutional Forward Projection }
	     \hspace*{-0.8cm}\textbf{Input:} \begin{minipage}[t]{0.8\textwidth}
 $N_x\times N_x$ phantom array $\overline f$, angle index $q\in [N_\phi]$   and detector index $p\in [N_s]$	
\end{minipage}	     
\\ 
\hspace*{-6.3cm} \textbf{Output: $(A\overline f)[qp]$ according to \eqref{equ_def_convolutional_Radon_discrete}}
    \label{algo:Pixel-Driven-Radon}
    \begin{algorithmic}[1] 
        \Function{Forwardprojection $(\overline f,q,p)$}{}
            				\State $\text{val} \gets 0$                
				 \For{ $j\in [N_x]$} \label{line:middle_forloop}
					\For {$i \in \mathcal{X}_{qp}^j:= \{i\in [N_x] \ \big|  \ x_{ij}\cdot\vartheta_q-s_p \in \supp{\omega(\phi_q,\cdot)}\}$} \label{line:inner_forloop}
					\State $\text{val}\gets \text{val}+\omega(\phi_q,x_{ij}\cdot \vartheta_q-s_p)\ \overline f[ij]$
					\EndFor
				\EndFor							
            \State \textbf{return} $\delta_x^2\text{val}$
        \EndFunction
    \end{algorithmic}
\end{algorithm}
\begin{algorithm}
    \caption{Convolutional Backprojection}
    \label{algo:Pixel-Driven-backprojection}
      \hspace*{-1.2cm}   \textbf{Input:} $N_\phi\times N_s$ sinogram array $\overline g$, spatial indices $i$ and $j$ in $[N_x]$ \\
  \hspace*{-6.4cm} \textbf{Output: $(B\overline g)[ij]$ according to \eqref{equ_def_convolutional_backproj_discrete}}
    \begin{algorithmic}[1] 
        \Function{Backprojection $(\overline g,i,j)$}{}
            				\State $\text{val} \gets 0$                
				 \For{ $q\in [N_\phi]$} \label{line:middle_forloop}
					\For {$p \in \mathcal{Y}_{ij}^q:= \{p\in [N_s] \ \big|  \ x_{ij}\cdot\vartheta_q-s_p \in \supp{\omega(\phi_q,\cdot)}\}$} \label{line:inner_forloop}
					\State $\text{val}\gets \text{val}+|\Phi_q|\ \omega(\phi_q,x_{ij}\cdot \vartheta_q-s_p)\ \overline g[qp]$
					\EndFor
				\EndFor							
            \State \textbf{return} $\delta_s\text{val}$
        \EndFunction
    \end{algorithmic}
\end{algorithm}

In principle, these algorithms are nothing more than matrix-vector multiplications. However, a key step is the determination of non-zero matrix entries (for the sake of efficiency) that is achieved via the sets $\mathcal{X}_{qp}^j$ and $\mathcal{Y}_{ij}^q$ in the algorithms' lines 4. If the weight function $t\mapsto\omega(\phi_q,t)$ has connected support $[\underline c_q,\overline c_q]$ (as we have for $\RayWeight$ and $\PixelWeight$), the relevant pixels $\mathcal{X}_{qp}^j$ can be determined efficiently.
Indeed, $\mathcal{X}_{qp}^j=[\underline i,\overline i] \cap [N_x]$ with
\begin{equation}
(\underline i, \overline i) = \text{sort}\left\{ \frac{ \frac {(\underline c_q +s_p- \bold y_j \cdot \vartheta_{\bold y} )}{ \vartheta_{\bold x}} +1}{\delta_x} -\frac 12 \ ,\ \frac{ \frac {(\overline c_q +s_p- \bold y_j \cdot \vartheta_{\bold y} )}{ \vartheta_{\bold x}} +1}{\delta_x} -\frac 12 \right \},
\end{equation}
where $\vartheta_q = (\vartheta_{\bold x}, \vartheta_{\bold y})$ denotes the two components of the projection direction and $x_{ij}=(\bold x_i, \bold y _j)$ with $\bold y_j = (j+\frac{1}{2})\delta_x-1$. Note that the formula only works if $\vartheta_{\bold x}\neq 0$. Similar to the method described in \cite{doi:10.1118/1.4761867},  one can swap the roles of $\bold x$ and $\bold y$ (and $i$ and $j$) if $|\vartheta_{\bold x}|<\sqrt{\frac{1}{2}}$.

Analogously, the set $\mathcal{Y}_{ij}^{q}$ can be identified (for $\omega$ with connected support) via $\mathcal{Y}_{ij}^{q} = [\underline p,\overline p] \cap [N_s]$ with
\begin{equation}
(\underline p,\overline p) = \left (\frac{x_{ij}\cdot\vartheta_q - \overline c_q+1}{\delta_s}-\frac{1}{2}\ ,\ \frac{x_{ij}\cdot\vartheta_q - \underline c_q+1}{\delta_s}-\frac{1}{2} \right).
\end{equation} 

As described in \cite{10.1007/s00245-022-09933-5}, the set $\mathcal{X}_{qp}^j$ for the ray-driven method with $\vartheta_{\bold x}>\frac{1}{\sqrt{2}}$ has at most two entries. Similarly, as described in \cite{doi:10.1137/20M1326635}, the set $\mathcal{Y}_{ij}^q$ has at most two entries for the pixel-driven method (see also \eqref{equ_lemma_relevant_p_pd}). This is of no immediate algorithmic advantage, but might have indirect advantages in terms of memory access.
Note that the computational complexities of different convolutional discretizations only differ by the complexity of evaluating weight functions $\omega$.

\subsection{Numerical experiments}

In this section, we perform numerical experiments to complement the presented
theoretical results. To that end, we executed pixel-driven projections using the
Gratopy toolbox \cite{kristian_bredies_2021_5221443}, 
while a custom implementation of the ray-driven method (as
an extension of Gratopy using Algorithms \ref{algo:Pixel-Driven-Radon} and \ref{algo:Pixel-Driven-backprojection}) was employed. The calculations were executed on a 12th Gen Intel(R)
Core(TM) i7-12650H processor in parallel with single precision. Throughout this
section, angles are chosen uniformly distributed, i.e., $\phi_q =  \frac{q}{N_\phi }\pi$ for $q \in [N_\phi]$. 
The
corresponding code is available via the GITHUB repository \cite{Huber_Github_2025_Simulations_Simulations_Convergence_Ray_Pixel_Driven_Strong_Topology}.

A number of numerical simulations concerning the approximation properties of convolutional discretizations were already presented in the author's conference paper \cite[Section 4]{10.1007/978-3-031-92366-1_11}, which is recommended for supplementary reading. The phantom and sinograms considered in those experiments were quite simple (an ellipse phantom and constant/linear sinograms), but the results are nonetheless illustrative.
Moreover, numerical experiments concerning \ref{equ_thm_pixel_radon} have already been presented in \cite{doi:10.1137/20M1326635}.

In the following, we consider piecewise constant functions $f=(f_{ij})_{ij}$ and $g=(g_{qp})_{qp}$ (as described above $f_{ij}$ and $g_{qp}$ refer to the pixel value associated with the pixel with the same indices) and for practical calculations associate them with vectors $\overline f$ and $\overline g$. 
We use the standard $\ell^2$ norms $\|g\|_{\ell^2([N_\phi]\times[N_s])}^2:= \delta_\phi \delta_s\sum_{q=0}^{N_\phi-1}\sum_{p=0}^{N_s-1} |g_{qp}|^2$ and $\|f\|_{\ell^2([N_x]^2)}^2:= \delta_x^2\sum_{i,j=0}^{N_x-1} |f_{ij}|^2$, where $\delta_s=\frac{2}{N_s}$, $\delta_x =\frac{2}{N_x}$ and $\delta_\phi = \frac{\pi}{N_\phi}$. It will be convenient for us to consider individual projections $g_q= (g_{qp})_p$ for a specific angular index $q$, and correspondingly $\|g_q\|_{\ell^2([N_s])}^2=\delta_s\sum_{p=0}^{N_s-1}|g_{qp}|^2$.
 Since all our sampling is uniform, these norms corresponds to the $L^2$ norm of the associated piecewise constant functions we described in the theory sections. In slight abuse of notation, $\RayRadon$, $\RayRadon^*$, $\PixelRadon$ and $\PixelRadon^*$ will accept and output arrays of the appropriate dimensions.

To quantitatively evaluate the performance of projections and backprojections, one would want to calculate the $L^2$ norms of differences between `true' projections and the described discretizations.  Thus, we define the $L^2$ error metrics
\begin{align}
&E(g,\tilde g) := \frac{\| g-\tilde g\|_{\ell^2([N_\phi]\times[N_s])}}{\|g\|_{\ell^2([N_\phi]\times[N_s])}} \qquad \text{and} \qquad E(f,\tilde f) := \frac{\| f-\tilde f\|_{\ell^2([N_x]^2)}}{\|f\|_{\ell^2([N_x]^2)}}. \notag 
\end{align}
Moreover, we define the angle-specific $L^2$ error and the worst angle $L^2$ error
\begin{align}
E_{\phi_q}(g,\tilde g) := \frac{\| g_{q}-\tilde g_{q}\|_{\ell^2([N_s])}}{\|g_q\|_{\ell^2([N_s])}} \qquad \text{and} \qquad
E_\text{max}(g,\tilde g) := \max_{q \in [N_\phi]} E_{\phi_q} (g,\tilde g). \notag
\end{align}

\subsubsection{Numerical experiment on the forward operator}
Here, we extend the experiments performed in \cite{10.1007/978-3-031-92366-1_11} on simple phantoms by using the FORBILD head phantom \cite{yu2012simulation}, which is more representative of real-world medical applications. It is a more complex phantom, containing a large number of ellipses (also intersected with half planes); see Figure \ref{Fig_numerics_phantom} a). We assume the phantom to occupy a 25cm$\times$25cm square and consider a detector with 25cm width  such that $s\in[-12.5,12.5]$ with $s=0$ corresponds to straight lines passing through the center of the phantom's square.
Naturally, this setting is equivalent (up to scaling) to the normalized setting described in Section \ref{section_Radon_notation}.
For now, let us fix $N_x=4096$, $N_s=4096$ (a balanced resolutions setting) and $N_\phi=1800$.
We used the code \cite{Forbildgen} (see \cite{yu2012simulation} for its documentation) to access the `discrete' FORBILD head phantom $f=(f_{ij})_{ij}$.  Note that the discrete phantom is a pointwise evaluation of the analytical representation (via ellipses, etc.), and not the mean values of the analytical phantom as proposed for convolutional discretizations in \eqref{equ_def_convolutional_Radon_transform}. The code also allows for the exact pointwise evaluation of the analytical phantom's Radon transform, depicted in Figure \ref{Fig_numerics_phantom} b). We denote the pointwise evaluations of these analytical projections by $g=(g_{qp})_{qp}$. The exact Radon transform of the discrete phantom does not necessarily coincide perfectly with the analytical Radon transform (of the analytical phantom). Since we use a very high resolution, the differences are, however, minimal and should not create any significant issues. Hence, we tacitly use the evaluations of the analytical Radon transform $g$ as ground truth below.

\begin{figure}[t]
\center
\newcommand{\mycolor}{red}
\newcommand{\myheight}{0.35}
\begin{overpic}[height=\myheight \textwidth]{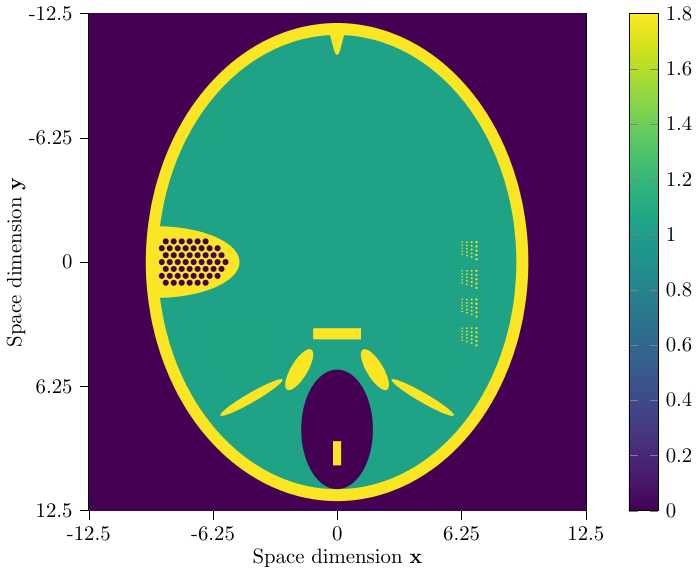}
 \put(15,84){\color{\mycolor} a) \color{black} \tiny FORBILD Head Phantom $f$}
 \end{overpic}
\begin{overpic}[height=\myheight \textwidth,rotate = 90]{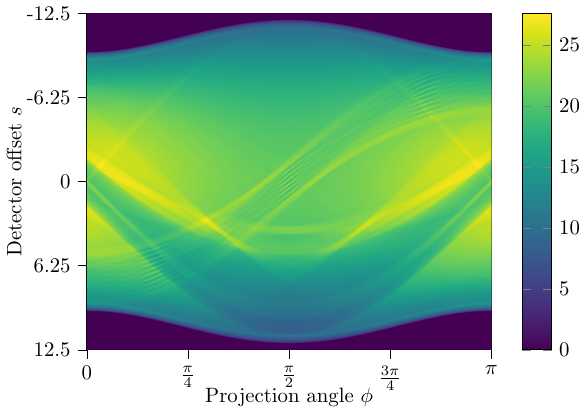}
 \put(15,71){\color{\mycolor} b) \color{black} \tiny Analytical projection $[\Radon f](\phi,s)$}
 \end{overpic} 
\caption{Depiction of the discrete $4096\times4096$ pixel FORBILD head phantom placed in the $[-12.5,12.5]^2$ square in a) and the corresponding $4096\times1800$ analytical Radon transform in b).}
\label{Fig_numerics_phantom}
\end{figure}

The corresponding ray-driven and pixel-driven Radon transforms (of the discrete phantom) are visually identical to the analytical projection. However, upon closer inspection, there are structural differences between the methods. Figure \ref{Fig_pointwise_errors} depicts the absolute differences of these two methods to the analytical projection (where, for visibility, we clipped the more extreme values; note the different scales between Figure \ref{Fig_numerics_phantom} b) and Figure \ref{Fig_pointwise_errors}).
Overall, both methods did very well. Larger errors are seen related to the finer structures in the phantom, which is to be expected.
However, there are some angles (most notably $\frac{\pi}{4}=45^\circ$ and $\frac{3\pi}{4}=135^\circ$), for which the pixel-driven projections appear very poor on the entire detector (see the vertical `streaks' in Figure \ref{Fig_pointwise_errors} b)), while the errors appear much more consistent in the ray-driven method. 

\begin{figure}
\vspace*{0.5cm}
\center
\newcommand{\myheight}{0.47}
\newcommand{\mycolor}{red}
\begin{overpic}[width=\myheight \textwidth]{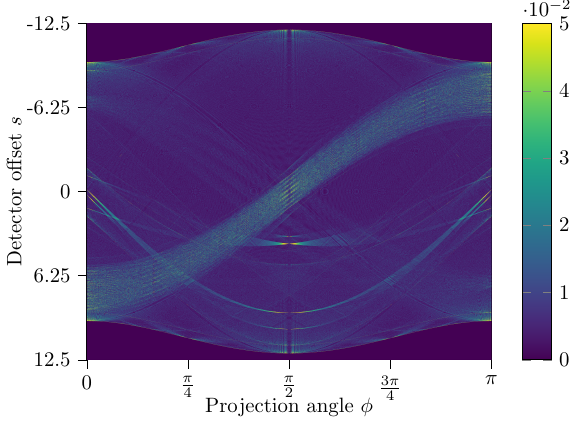}
 \put(10,75){\begin{minipage}{0.35\textwidth}\center
 \color{\mycolor} a) \color{black} \tiny  Ray-driven absolute difference 
 \\
 $\left|[(\Radon-\RayRadon) f](\phi,s)\right|$\end{minipage}}
 \end{overpic}
\begin{overpic}[width=\myheight \textwidth,rotate = 90]{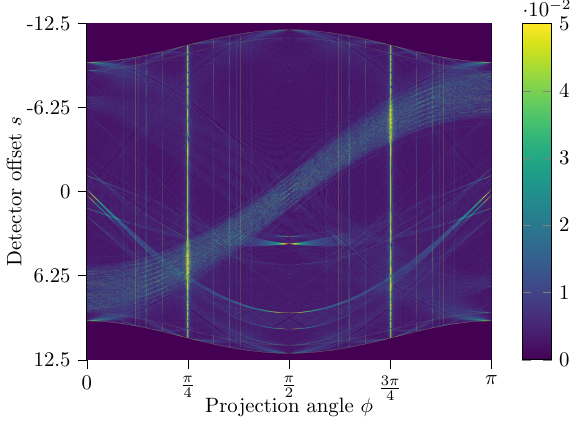}
 \put(10,75){\begin{minipage}{0.35\textwidth}\center
 \color{\mycolor} a) \color{black} \tiny  Pixel-driven absolute difference 
 \\
 $\left|[(\Radon-\PixelRadon) f](\phi,s)\right|$\end{minipage}} 
 \end{overpic}
 \caption{Illustration of the pointwise absolute difference between the analytical Radon transform of the FORBILD phantom and the ray-driven projection (on the left) or the pixel-driven projection (on the right). For both, the balanced situation $N_x=N_s=4096$ and $N_\phi = 1800$ is used.}
 \label{Fig_pointwise_errors}
 
\end{figure}

To illustrate the bad projections' behavior, Figure \ref{Fig_numerics_plot_worst_projection} depicts the  projections for the angles $\phi= 135^\circ$ and $\phi=135.1^\circ$. As can be seen, the pixel-driven projection creates very significant oscillations for the former, while the change by a tenth of a degree reduces these oscillations significantly. This can also be observed in the relative $L^2$ error for these projections reducing from $E_{135^\circ}(g,\PixelRadon f) \approx 6.6\%$ to $E_{135.1 ^\circ}(g,\PixelRadon f) \approx0.5\%$.

As discussed in \cite{doi:10.1137/20M1326635} and \cite{huber2022pixel}, multiples of $\frac{\pi}{4}$ appear to be prime suspects for such bad projections in pixel-driven Radon transforms, and the oscillations reduce with unbalanced resolutions $\frac{\delta_x}{\delta_s}\to 0$. \cite[Example III.2.33]{huber2022pixel} gives an analytical explanation of why some angles produce such bad projections although one would think these projections `easier' than others. A naive explanation is that if $\frac{\delta_x}{\delta_s}=c$, then the weights $\PixelWeight(x_{ij}\cdot \vartheta_q-s_p)$ do not attain a meaningful value for each individual pixel $X_{ij}$, but since the values $x_{ij}\cdot \vartheta_q-s_p$ are approximately `uniformly distributed' for most $q$, they achieve the right value in the mean of all $ij\in [N_x]^2$. However, when the direction $\vartheta_q$ is for example upwards or diagonal, the pixels $X_{ij}$ appear much more aligned, and the suggested uniform distribution is broken.
If we knew a priori which angles were bad, a minuscule angular shift (thus avoiding the poor projections) could be a strategy to remedy the oscillations.

\begin{figure}

\newcommand{\mycolor}{teal}
\newcommand{\myheight}{0.4}
\begin{overpic}[height=\myheight \textwidth]{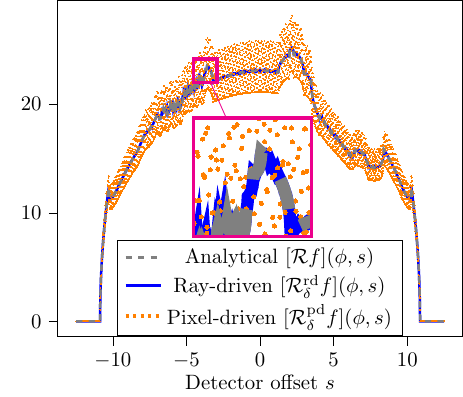}
 \put(15,79){\color{\mycolor} a) $\phi=135.0^\circ$}
 \end{overpic}
 \begin{overpic}[height=\myheight \textwidth]{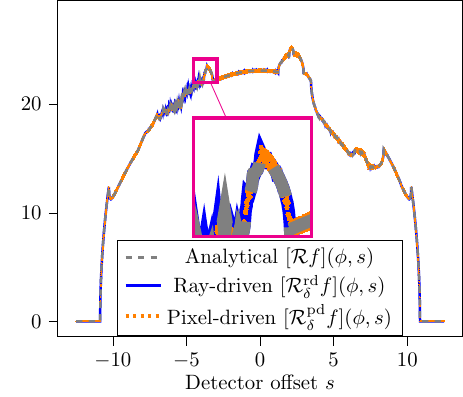}
 \put(15,79){\color{\mycolor} b) $\phi=135.1^\circ$}
 \end{overpic}

\caption{Illustration of the projections for $\phi=\frac{3}{4}\pi=135^\circ$ on the left, and for $\phi = 135.1^\circ$ on the right. Even though the angles only differ by a tenth of a degree, the pixel-driven method reveals significant errors in the former setting that have all but disappeared in the latter one.}
\label{Fig_numerics_plot_worst_projection}

\end{figure}

Figure \ref{Fig_numerics_error_per_projection} plots the angle-specific relative $L^2$ errors $E_\phi(g,\PixelRadon f)$ of each individual projection angle $\phi$ to illustrate the presence of these outlier projections further. As can be seen, for many projection angles, the pixel-driven Radon transform incurs smaller errors than the ray-driven Radon transform, which in itself is surprising given the method's reputation (and \eqref{equ_thm_pixel_radon}). Besides the mentioned $45^\circ$ and $135^\circ$ projection angles, there are a few other (less severe) outliers in the pixel-driven projection. 
In particular, the relative error of $0.5\%=5\ 10^{-3}$ for $135.1^\circ$ we mentioned above still exceeds average errors (of roughly $0.05\%$) by a factor of $10$. In contrast, the errors of the ray-driven method appear much more uniform for all angles.

\begin{figure}
\center
\includegraphics[width=0.9\textwidth]{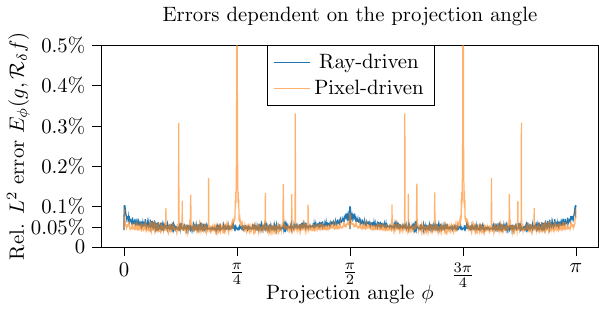}
\caption{Illustration of the relative $L^2$ error of the individual projections (i.e., for variable $\phi$) $E_{\phi_q}(g,\RayRadon f)$ and $E_{\phi_q}(g,\PixelRadon f)$. The extreme cases $\frac{\pi}{4}$ and $\frac{3\pi}{4}$ are far beyond the shown scale (by a factor of 10), but there are also other projections whose errors far exceed the average errors (around $5\ 10^{-4}$).}

\label{Fig_numerics_error_per_projection}
\end{figure}


So far, we have only considered the fixed resolutions $N_x = N_s = 4096$ and $N_\phi = 1800$. However, for us, the dependence of errors on the discretization parameters is of particular importance. To that end, in Figure \ref{Fig_numerics_evolution_errors} a), we illustrate the evolution of the relative $L^2$ errors $E(g,\RayRadon f)$ and $E(g,\PixelRadon f)$ of the ray-driven and pixel-driven methods in a balanced resolutions setting with fixed $N_\phi=360$ and increasing $N_x=N_s$. It appears that the ray-driven error moves towards zero with increasingly finer resolutions, while the pixel-driven's error appears to stagnate (or slow down significantly) long before reaching zero. Note that this is consistent with the theory, suggesting convergence for the ray-driven method (see Corollary \ref{thm_convergence_sparse_angle}), while for the pixel-driven method, our theory does not offer implications. 
Figure \ref{Fig_numerics_evolution_errors} b) plots the corresponding relative $L^2$ errors of the worst projection $E_\text{max}(g,\Radon_\delta f)$. As can be seen, this error is stagnant for the pixel-driven method, while it decreases significantly for the ray-driven method.
This suggests that the root cause of the oscillations in the pixel-driven setting remains even if we refine the balanced resolution.

Note that we also plotted the error for the pixel-driven projection with only $180$ projections in Figure \ref{Fig_numerics_evolution_errors} a),  showing a similar stagnation but at a higher error (roughly by a factor of $\sqrt{2}$). The individual projections' quality remained the same, but having fewer projections in total, the outliers contributed proportionally more (by a factor of $\sqrt{2}$) to the overall error in the sinogram.

\begin{figure}
\center
\newcommand{\mycolor}{red}
\newcommand{\mywidth}{0.85}
\vspace*{0.25cm}
\begin{overpic}[width=\mywidth \textwidth]{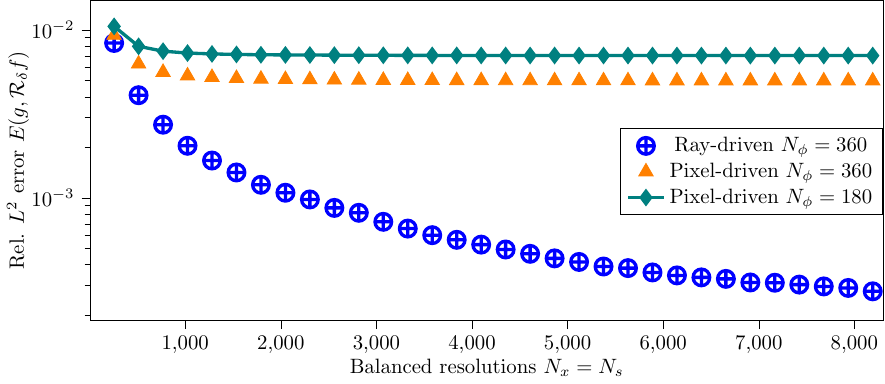}
\put(10,45){\color{\mycolor} a) \color{black} $L^2$ errors for the entire sinogram}
\end{overpic}\vspace{0.5cm}

\begin{overpic}[width=\mywidth \textwidth]{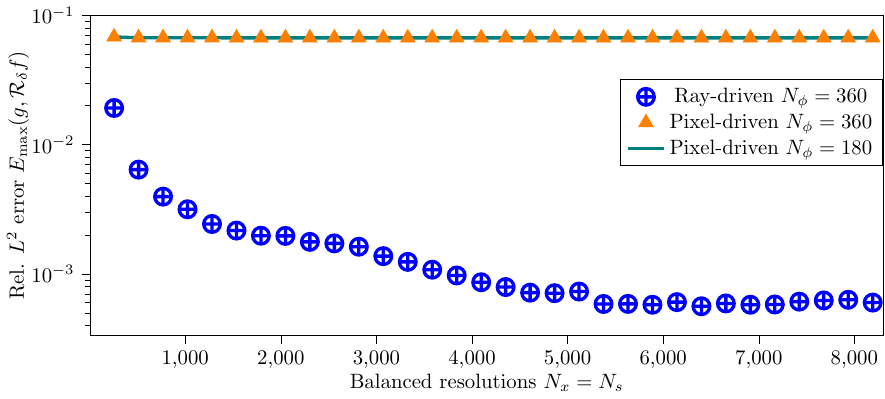}
\put(10,45){\color{\mycolor} b) \color{black}   Worst projection's $L^2$ errors }
\end{overpic}
\caption{Illustration of the relative $L^2$ errors for fixed $N_\phi=360$ and increasing balanced resolutions $N_x=N_s$ between $256$ and $8192$ in 32 steps. On the top, the errors $E(g,\Radon_\delta f)$ for the entire sinogram are depicted, while the figure below shows the relative error of the worst single projection $E_\text{max}(g,\Radon_\delta f)$. As can be seen, the pixel-driven error appears stagnant.}
\label{Fig_numerics_evolution_errors}

\end{figure}


\subsubsection{Numerical experiments on the backprojection}
We consider the sinogram $g(\phi,s)= 1$ for all $(\phi,s)\in \sinodom$. It is trivial to calculate that $[\Radon^* g ](x) = \pi$ for all $x \in \imgdom$ from the definition \eqref{equ_def_backprojection}. We assume that in the following all angles $\phi_q$ are sampled equidistantly in $[0,\pi[$.

Calculating the backprojections explicitly for an image pixel $ij$, we see
\begin{equation}\label{equ_numerics_backprojection_ray_driven_definition}
[\RayRadon^* g]_{ij} \overset{\eqref{equ_def_convolutional_backproj_discrete}}{=} \sumphi |\Phi_q| \delta_s \sums \RayWeight(\phi_q,x_{ij}\cdot \vartheta_q-s_p),
\end{equation}
and 
\begin{equation}\label{equ_numerics_backprojection_pixel_driven_definition}
[\PixelRadon^* g]_{ij} \overset{\eqref{equ_def_convolutional_backproj_discrete}}{=}  \sumphi |\Phi_q| \delta_s\sums \PixelWeight(\phi_q,x_{ij}\cdot \vartheta_q-s_p).
\end{equation}
Using \eqref{equ_lemma_pixel_sum_s} on \eqref{equ_numerics_backprojection_pixel_driven_definition}, we note that the internal sum 
\[\PixelFq(ij):= \delta_s\sums \PixelWeight(\phi_q,x_{ij}\cdot \vartheta_q-s_p)\] equals $1$ if $|x_{ij}|\leq 1-  \frac1 {\delta_s}$ (and consequently $|x_{ij}\cdot \vartheta_\phi|\leq 1 -\frac{1}{\delta_s}$ for all $\phi\in [0,\pi[$).
Moreover, $\sumphi |\Phi_q| = \pi$ implies that $[\PixelRadon^* g]_{ij} = \pi$, i.e., the pixel-driven backprojection attains the correct value $\Radon^*g=\pi$, for aforementioned $x_{ij}$.  Thus, with the exception of a small ring, all pixels in the unit ball will attain the correct value via the pixel-driven method. Note that we will later only consider and plot $x_{ij}$ with $|x_{ij}|\leq 0.95$ (thus avoiding the mentioned outermost pixels) as the behavior at the boundary is not central to our considerations here.

In contrast, in the ray-driven method, we observe that the innermost sum 
\[\RayFq (ij) = \delta_s \sums \RayWeight(\phi_q,x_{ij}\cdot \vartheta_q-s_p)\] of \eqref{equ_numerics_backprojection_ray_driven_definition} is not guaranteed to achieve the value $1$ (as was the case in the pixel-driven setting). As described in \eqref{equ_lemma_ray_sum_p}, when $\frac{\delta_s}{\delta_x} \to 0$, also $\RayFq(ij)$ will converge to $1$ for each $q$ and $ij$ uniformly when $|x_{ij}|\leq 1-\frac{\delta_x}{2}$ (the other cases will again be excluded from our numerical investigations below). In contrast, when $\delta_x\approx \delta_s$, then the values of $\RayFq(ij)$ -- being the Riemann sums described in the Proof of Lemma \ref{Lemma_sums_weights} -- cannot be expected to equal one.  In \eqref{equ_numerics_backprojection_ray_driven_definition} we still average with respect to $\phi_q$, so even if $\RayFq(ij)\neq 1$ for some $q$ (and fixed $ij$), it might well be that $\sumphi |\Phi_q| \RayFq(ij) \approx \pi$; one could say the ray-driven method does not do the right thing for individual angles but on average.  
And indeed, in numerical experiments, one can observe such behavior, although one needs to mention that the related convergence is very slow (see Figure \ref{Fig_numerics_backprojection_angle_decay}). While this effect occurs for this highly trivial $g=1$ example, this might not imply the same behavior for more complex sinograms $g$.

In conclusion, from our theoretical investigations we can already see that the pixel-driven method should be precise for this example, while for the ray-driven backprojection method, the error should chiefly depend on $\frac{\delta_s}{\delta_x}$ (as this mainly influences \eqref{equ_lemma_ray_sum_p}).

For numerical experiments, we start by considering increasing $N_x$ for a fixed number of angles $N_\phi=90$ and three different choices of $\frac{N_x}{N_s}=\frac{\delta_s}{\delta_x}\in \left \{1,\frac{1}{2},\frac{1}{4}\right\}$. As expected, for the pixel-driven method, the error is within the machine precision (single precision, roughly $10^{-6}$) and presumably is only non-zero due to rounding errors.
In contrast, when visualizing the error in the $\frac{\delta_s}{\delta_x}=1$ setting in Figure \ref{Fig_numerics_backprojection_evolution_balanced}, the error appears to be stagnant, i.e., even with greatly increased $N_x=N_s$, the error does not decrease. That is the case as $\RayFq(ij)$ chiefly depends on $\frac{\delta_s}{\delta_x}$ which remains fixed. In contrast, we also plotted the error with $\frac{\delta_s}{\delta_x}=\frac{1}{2}$ and $\frac{\delta_s}{\delta_x}=\frac{1}{4}$, showing that the errors are again stagnant but at significantly reduced levels; validating \eqref{equ_thm_ray_estimate_backprojeciton}.

\begin{figure}
\includegraphics[width = 0.9 \textwidth]{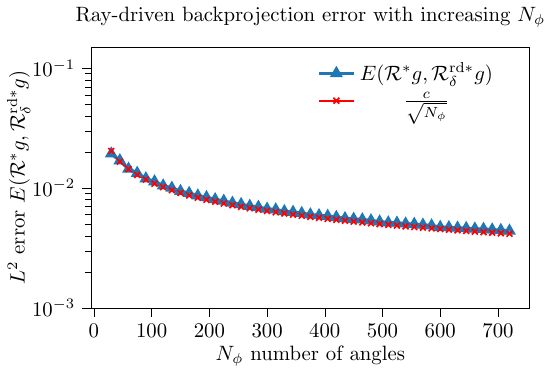}
\caption{Development of relative ray-driven backprojection errors for fixed balanced resolution $N_x=N_s=2000$ and increasing $N_\phi$. The behavior is roughly $\sqrt{\delta_\phi}$; see the red crosses for comparison.}
\label{Fig_numerics_backprojection_angle_decay}
\end{figure}

\begin{figure}
\vspace*{0.5cm}
\begin{overpic}[width=0.99\textwidth]{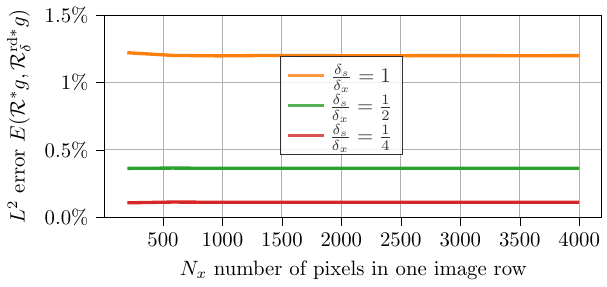}
\put(12,48){Evolution of ray-driven backprojection's errors for fixed ratios $\frac{\delta_s}{\delta_x}$}
\end{overpic}
\caption{Illustration of the error $E(\Radon^*g,\RayRadon^*g)$ for the fixed $N_\phi=90$ and increasing $N_x$ and three different ratios $\frac{\delta_s}{\delta_x}=\frac{N_x}{N_s}$.}
\label{Fig_numerics_backprojection_evolution_balanced}
\end{figure}

It might appear curious that even with significantly increased computational effort (with increased $N_x=N_s$), the errors remain the same, while reducing $N_x$ (and thus also the computational effort) reduces the error.
To illustrate the behavior of the calculated projections, in Figure \ref{Fig_numerics_backprojection_different_parameters} we   depict some of the related backprojections. As can be seen in the case $\frac{\delta_s}{\delta_x}=1$, the ray-driven backprojection introduces highly oscillatory patterns to the backprojection that persist as the $N_x$ increase with higher frequencies and the same amplitudes. In the second row of Figure \ref{Fig_numerics_backprojection_different_parameters} with $\frac{\delta_s}{\delta_x}=\frac{1}{2}$, these oscillations do not disappear, but their amplitude has decreased significantly (note the difference in scales). In particular, counterintuitively, calculating the backprojection with $N_s=1000$ and $N_x=1000$ requires $4$ times more computational effort than using $N_x=500$, while creating an error that is roughly three times larger (1.20\% compared to 0.36\%).


\begin{figure}
\newcommand{\mycolor}{black}
\newcommand{\mysize}{0.32}
\vspace*{0.25cm}
\begin{overpic}[width=\mysize \textwidth]{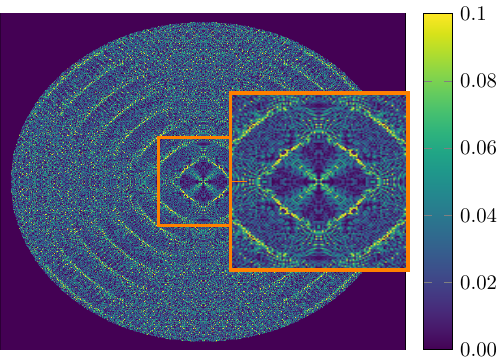}
 \put(0,73){\color{\mycolor} a)  \tiny $Nx=500$, $N_s=500$}
\end{overpic}
\begin{overpic}[width=\mysize \textwidth]{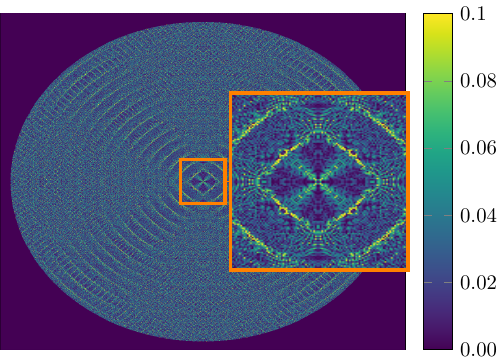}
 \put(0,73){\color{\mycolor} b)  \tiny $Nx=1000$, $N_s=1000$}
\end{overpic}
\begin{overpic}[width=\mysize \textwidth]{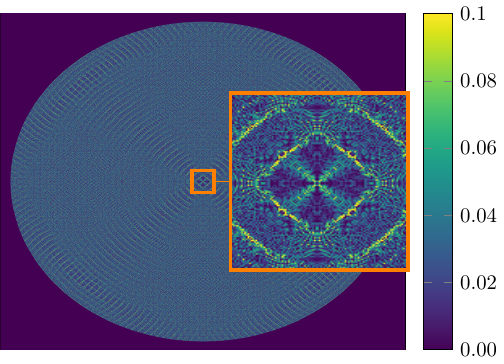}
 \put(0,73){\color{\mycolor} c)  \tiny $Nx=2000$, $N_s=2000$}
\end{overpic}\bigskip

\begin{overpic}[width=\mysize \textwidth]{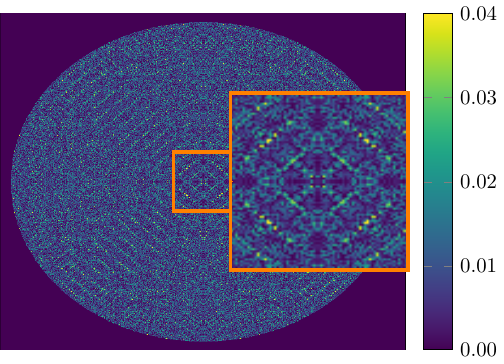}
 \put(0,73){\color{\mycolor} d)  \tiny $Nx=500$, $N_s=1000$}
\end{overpic}
\begin{overpic}[width=\mysize \textwidth]{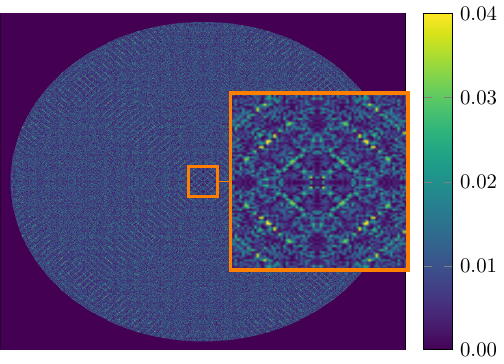}
 \put(0,73){\color{\mycolor} e)  \tiny $Nx=1000$, $N_s=2000$}
\end{overpic}
\begin{overpic}[width=\mysize \textwidth]{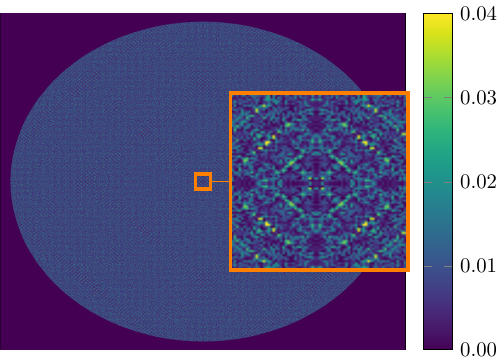}
 \put(0,73){\color{\mycolor} f)  \tiny $Nx=2000$, $N_s=4000$}
\end{overpic}
\caption{The first row shows the pointwise absolute difference $|\RayRadon^*g-\Radon^*g|$ for increasing balanced resolutions $\frac{\delta_s}{\delta_x}=1$. The second line shows the same differences for $\frac{\delta_s}{\delta_x}=\frac{1}{2}$. In both cases $N_\phi=90$ remains fixed. Note the different colormaps between the first and second row.}
\label{Fig_numerics_backprojection_different_parameters}
\end{figure}

As mentioned, the theory suggests that as $\frac{\delta_s}{\delta_x}$ goes to zero, the approximation of the backprojection gets better both in theory (see Theorem \ref{Thm_approximation_ray_driven}) but also concretely as we discussed above as $\RayFq(ij)\to 1$.
To further illustrate this point, in Figure \ref{Fig_numerics_backprojection_ratios}, we plot for fixed $N_x=1000$ (or fixed $N_x=2000$), the evolution of the errors for decreasing $\frac{\delta_s}{\delta_x}$ (in the plot, the $\bold x$ axis is the inverse for better visibility). The errors overall decrease as the said ratio decreases. A curious effect is the seemingly periodic structure of the error, where the ray-driven backprojection appears to perform best at integer multiples of $\frac{\delta_x}{\delta_s}\in \NN$. The author does not have a good explanation for this effect at this point beyond suspecting that some symmetry effects in $\RayFq$ make calculations more accurate by canceling out errors. This might be the topic of future investigations.
Note that considering the same experiment with $N_x=2000$ yields virtually the same results, further suggesting that indeed the ratio $\frac{\delta_s}{\delta_x}$ is the major factor in the errors.

\begin{figure}
\vspace*{0.25cm}
\begin{overpic}[width=0.99\textwidth]{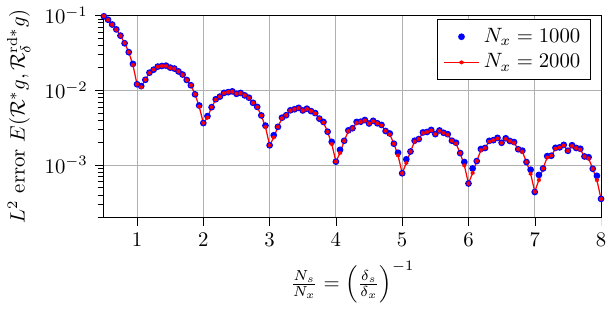}
\put(15,50){Evolution of errors for increasing $N_s$, fixed $N_x$ and $N_\phi=90$}
\end{overpic}
\caption{The evolution of relative $L^2$ errors $E(\Radon^*g,\RayRadon^*g)$ in the ray-driven backprojection with increasing $N_s$ while $N_\phi=90$ is fixed and $N_x=1000$ or $N_x=2000$.}
\label{Fig_numerics_backprojection_ratios}
\end{figure}

\section{Conclusion and Outlook}
This paper presented an interpretation of the ray-driven and pixel-driven discretization frameworks as finite rank operators via convolutional discretizations with ray- and pixel-driven weight functions. This interpretation allowed us to prove corresponding convergence statements in the strong operator topology (i.e., pointwise) in Theorem \ref{Thm_approximation_ray_driven}. This result gives a theoretical foundation to the widespread use of ray-driven forward and pixel-driven backprojection operators under balanced resolutions, confirming anecdotal reports concerning ray-driven Radon transforms and pixel-driven backprojections being suitable approximations. 

For the ray-driven backprojection and pixel-driven Radon transform, we did not present any convergence results in the balanced resolution setting. If they indeed do not converge, this makes the use of the matched approaches $\mathrm{rd}-\mathrm{rd}^*$ and $\mathrm{pd}-\mathrm{pd}^*$ questionable, as at least one of the used discretizations does not converge. 
In contrast, the combination of ray-driven and pixel-driven methods ($\mathrm{rd}$-$\mathrm{pd}^*$) guarantees that all the used discretizations converge in the balanced resolution setting. While using unmatched operators has its own set of issues, employing convergent unmatched discretizations ($\mathrm{rd}$-$\mathrm{pd}^*$) might be preferable to non-convergent matched discretizations ($\mathrm{rd}$-$\mathrm{rd}^*$) and ($\mathrm{pd}$-$\mathrm{pd}^*$).

 Related line integral operators like the fanbeam- and conbeam transforms (modeling similar tomography scenarios) also possess ray- and pixel-driven discretizations.
 The convergence results of this paper can probably be transferred in a straightforward manner to these other operators' discretizations, which might be the topic of future investigations.

The convergence result \eqref{equ_thm_ray_estimate_backprojeciton} shows that, while the ray-driven backprojection might be unsuitable for balanced resolutions, it indeed approximates the backprojection in the unbalanced resolution setting $\frac{\delta_s}{\delta_x}\to 0$. Since both the forward and backprojection converging in the strong operator topology is a prerequisite for convergence in the operator norm, it might be that in this unbalanced resolution setting, the ray-driven Radon transform converges in the operator norm.
Future work might investigate such convergence properties. 

Many of the theoretical properties described in Sections \ref{section_Radon} and \ref{section_Radon_proofs} are also observed in our numerical experiments in Section \ref{section_numerics}. That the pixel-driven Radon transform outperformed the ray-driven Radon transform for many projection angles (see Figure \ref{Fig_numerics_error_per_projection}) was a bit surprising, and might warrant further investigation.

While it seems intuitive that better discretizations should result in reconstructions being more faithful to the infinite-dimensional tomography problem (using the continuous Radon transform), there appears to be a gap in the literature not rigorously arguing that point. Bridging that gap will be a key step in connecting `discrete' inverse problems solved on computers with `infinite-dimensional' inverse problems that are commonly investigated on a theoretical level.

\appendix

\section*{Appendix: Proof of Lemma \ref{Lemma_weight_as_intersection_length}}
We fix some $\phi \in [0,\pi[$ and $s\in \RR$, and set $\underline s = \underline s(\phi)$ and $\overline s = \overline s(\phi)$ for the sake of readability. 
Recall that $|s|$ describes the normal distance of $L_{\phi,s}$ to the origin $(0,0)\in \RR^2$, and a point $x$ satisfies $x\in L_{\theta,s}$  if and only if $x\cdot \vartheta_\phi=s$. In particular, all points $x$ with $x\cdot \vartheta_\phi>s$ are on one side of $L_{\phi,s}$, and all $x$ with $x\cdot \vartheta_\phi<s$ on the other.

We consider the square $Z :=\left[-\frac{\delta_x}{2},\frac{\delta_x}{2} \right ]^2$. Our first goal is to show
\begin{equation} \label{equ_proof_intersectionlengths_via_ray_weight}
\delta_x^2\RayWeight(\phi,s)=\mathcal{H}^1(L_{\phi,s}\cap Z)-\frac{1}{2}\mathcal{H}^1(L_{\phi,s}\cap \partial Z)=:F(\phi,s),
\end{equation}
from which \eqref{equ_lemma_intersectionlength} will easily follow.

We divide the calculation of \eqref{equ_proof_intersectionlengths_via_ray_weight} into multiple cases. Since the following considerations are quite geometric, see Figure \ref{Fig_intersection_lengths} for their visual representation.
Via straightforward calculation, the four corners $(\pm \frac{\delta_x}{2},\pm \frac{\delta_x}{2})\in \RR^2$ of $Z$ lie exactly on the lines associated with $-\overline s,$ $-\underline s,$ $\underline s$ and $\overline s$.
Therefore, if $|s|< \underline s$ (case 1), i.e., $-\overline s\leq - \underline s <s<\underline s\leq \overline s$, there are two vertices on either side of $L_{\phi,s}$. If $|s|\in [\underline s,\overline s[$ (case 2), then one side of $L_{\phi,s}$ has only a single corner. If $|s|>\overline s$ (case 3), then one side of $L_{\phi,s}$ does not contain any corners. This leaves the special cases $|s|=\underline s$ and $\phi \not \in \frac{\pi}{2}\ZZ$ (case 4), $|s|=\overline s$  and $\phi \not \in \frac{\pi}{2}\ZZ$ (case 5), and finally  $|s|=\overline s$ and $\phi \in \frac{\pi}{2}\ZZ$  implying $\underline s = \overline s$ (case 6).

Moreover, note that $L_{\phi,s}\cap \partial Z$ contains exactly two points when $|s|<\overline s$ and thus is an $\mathcal{H}^1$ null set. When $|s|=\overline s$, there is exactly one element in $L_{\phi,s}\cap \partial Z$ if $\phi \not \in \frac{\pi}{2}\ZZ$ (i.e., $L_{\phi,s}$ is not parallel to one of the sides of $Z$), and it is one entire side of $Z$ otherwise. Hence, $\mathcal{H}^1(L_{\phi,s}\cap \partial Z)$ is only non-zero in the case 6, and can otherwise be ignored.

\underline{Case 1:}
If $|s|<\underline s$, we consider the right triangle formed by the two points $z_1,z_2\in L_{\phi,s}\cap \partial Z$ (lying on opposite sides of $Z$) and one point exactly opposite $z_1$; see Figure \ref{Fig_intersection_lengths} b). Hence, one side of the triangle is $L_{\phi,s}\cap Z$ with length $r$, and one other side's length is $\delta_x$.
Since one of the triangle's sides is $\delta_x$ long and it contains the angle $\phi$, the hypotenuse's length equals $F(\phi,s)=r=\delta_x \min\left(\frac{1}{|\cos(\phi)|},\frac{1}{|\sin(\phi)|}\right)\overset{\text{per}}{\underset{\text{def}}{=}}\delta_x^2\RayWeight(\phi,s)$.

\underline{Case 2:}
When $|s|\in [\underline s,\overline s[$, the two points of $L_{\phi,s}\cap \partial Z$ form a right triangle with the single corners of $Z$ on one side of $L_{\phi,s}$. In particular, $L_{\phi,s}\cap Z$ forms the hypotenuse of said triangle with length $r$, while we denote the catheti's lengths by $a$ and $b$, and the height as $h$; see Figure \ref{Fig_intersection_lengths} c). We note that the height satisfies $h=\overline s -|s|$.
The area in a right triangle satisfies $\text{Area} = \frac{ab}{2} = \frac{rh}{2}$, which together with  $a=r|\sin(\phi)|$ and $b=r|\cos(\phi)|$  implies  
\begin{equation}
 F(\phi,s)= r = \frac{h}{|\cos(\phi)\sin(\phi)|}= \frac{\overline s -|s|}{|\cos(\phi)\sin(\phi)|}\overset{\text{per}}{\underset{\text{def}}{=}}\delta_x^2 \RayWeight(\phi,s).
\end{equation}

\underline{Case 3:}
Since $Z$ is the convex hull of its corners, $Z$ in its entirety lies on one side of $L_{\phi,s}$ if $|s|>\overline s$, and thus $L_{\phi,s}\cap Z=\emptyset$, implying $F(\phi,s)=0\overset{\text{per}}{\underset{\text{def}}{=}}\delta_x^2 \RayWeight(\phi,s)$.

\underline{Case 4:}
Since $\phi \not \in \frac{\pi}{2}\ZZ$ (implying $\underline{s}<\overline{s}$), $|s|=\underline s$ is precisely the same situation as described in case 2.

\underline{Case 5:}
In the case $\phi \not \in \frac{\pi}{2}\ZZ$ and $|s|=\overline s$, the intersection $L_{\phi,s}\cap Z$ contains only a single point (the corner), thus having Hausdorff measure zero, which coincides with $\delta_x^2\RayWeight(\phi,s)$ (this falls into the `else' case of \eqref{equ_def_ray_weight}).

\underline{Case 6:}
If $\phi \in \frac{\pi}{2}\ZZ$ and $|s|=\overline s$,  we have to take $\partial Z$ into account (this is the only case where the Hausdorff measure of $L_{\phi,s}\cap \partial Z$ is not zero). 
In particular, $L_{\phi,s}\cap Z=L_{\phi,s}\cap \partial Z$ is one side of $Z$, and thus $\mathcal{H}^1(L_{\phi,s}\cap Z)=\mathcal{H}^1(L_{\phi,s}\cap \partial Z)=\delta_x$. Therefore, $F(\phi,s)=\delta_x-\frac{\delta_x}{2}= \frac{\delta_x}{2}\overset{\text{per}}{\underset{\text{def}}{=}} \delta_x^2 \RayWeight(\phi,s)$.

In conclusion, \eqref{equ_proof_intersectionlengths_via_ray_weight} holds.
Since the Hausdorff measure is translation invariant, $X_{ij}=x_{ij}+Z$ and $L_{\phi,s}= x_{ij}+L_{\phi,s-x_{ij}\cdot\vartheta_\phi}$, we finally get
\begin{equation}
    \mathcal{H}^1(L_{\phi,s}\cap X_{ij})= \mathcal{H}^1\big((x_{ij}+L_{\phi,s-x_{ij}\cdot \vartheta_\phi})\cap (x_{ij}+Z )\big)
    =\mathcal{H}^1\big(L_{\phi,s-x_{ij}\cdot \vartheta_\phi}\cap Z \big)
    .
\end{equation}
and analogously for $\mathcal{H}^1(L_{\phi,s}\cap \partial X_{ij})$, resulting in \eqref{equ_lemma_intersectionlength}.

\begin{figure}
\newcommand{\myheight}{3}
\newcommand{\mydepth}{0.3}

\center
\newcommand{\figursize}{1.1}
\begin{tikzpicture}[scale=\figursize]

\draw[] (-\myheight,\myheight) node [right,yshift=-0.2cm]{a)};

\clip (-\myheight-0.8,-\mydepth) rectangle (0+\mydepth,\myheight+0.8) ;
\draw(-\myheight,0) -- (0,0);
\draw(0,0) -- (0,\myheight);
\draw(-\myheight,0) -- (-\myheight,\myheight);
\draw(0,\myheight) -- (-\myheight,\myheight);

\pgfmathsetmacro{\myangle}{25}
\pgfmathsetmacro{\t}{2}

\pgfmathsetmacro{\vx}{cos(\myangle)}
\pgfmathsetmacro{\vy}{sin(\myangle)}

\pgfmathsetmacro{\wx}{sin(\myangle)}
\pgfmathsetmacro{\wy}{-cos(\myangle)}


\pgfmathsetmacro{\supper}{0*\wx+\myheight*\wy}
\pgfmathsetmacro{\slower}{-\myheight*\wx+0*\wy}
\pgfmathsetmacro{\N}{5}
\pgfmathsetmacro{\n}{0}
{
\pgfmathsetmacro{\myxi}{\slower+(\supper-\slower)*\n/\N}
\pgfmathsetmacro{\tone}{-\myxi*\wy/\vy}
\pgfmathsetmacro{\ttwo}{-\myxi*\wx/\vx}
\draw[teal,ultra thick] (\tone*\vx+\myxi*\wx,\tone*\vy+\myxi*\wy) -- (\ttwo*\vx+\myxi*\wx,\ttwo*\vy+\myxi*\wy) node [midway,above,xshift=-0.1cm] {$L_{\phi,\underline s}$};
}


\pgfmathsetmacro{\s}{-\myheight*\wx+0*\wy}
\pgfmathsetmacro{\N}{7}

\pgfmathsetmacro{\s}{0*\wx+\myheight*\wy}

\pgfmathsetmacro{\n}{0}
{
\pgfmathsetmacro{\myxi}{\s-2*\n/\N}
\pgfmathsetmacro{\tone}{-\myxi*\wy/\vy}
\pgfmathsetmacro{\ttwo}{-\myxi*\wx/\vx}
\draw[teal,ultra thick] (\tone*\vx+\myxi*\wx,\tone*\vy+\myxi*\wy) -- (\ttwo*\vx+\myxi*\wx,\ttwo*\vy+\myxi*\wy) node [midway,above,xshift=1.5cm] {$L_{\phi,-\underline s}$};
}

\pgfmathsetmacro{\s}{-0.9}
\pgfmathsetmacro{\n}{0}
\pgfmathsetmacro{\myxi}{\s-2*\n/\N}
\pgfmathsetmacro{\tone}{-\myxi*\wy/\vy}
\pgfmathsetmacro{\ttwo}{-\myxi*\wx/\vx}
\draw[red,ultra thick] (\tone*\vx+\myxi*\wx,\tone*\vy+\myxi*\wy) -- (\ttwo*\vx+\myxi*\wx,\ttwo*\vy+\myxi*\wy) node [midway,below,xshift=0.2cm] {$L_{\phi,s_2}$};

\pgfmathsetmacro{\s}{-2}
\pgfmathsetmacro{\n}{0}
\pgfmathsetmacro{\myxi}{\s-2*\n/\N}
\pgfmathsetmacro{\tone}{-\myxi*\wy/\vy}
\pgfmathsetmacro{\ttwo}{-\myxi*\wx/\vx}
\draw[brown,ultra thick] (\tone*\vx+\myxi*\wx,\tone*\vy+\myxi*\wy) -- (\ttwo*\vx+\myxi*\wx,\ttwo*\vy+\myxi*\wy) node [midway,below,xshift=0.2cm] {$L_{\phi,s_1}$};

\pgfmathsetmacro{\s}{-4}
\pgfmathsetmacro{\n}{0}
\pgfmathsetmacro{\myxi}{\s-2*\n/\N}
\pgfmathsetmacro{\tone}{-\myxi*\wy/\vy+6}
\pgfmathsetmacro{\ttwo}{-\myxi*\wx/\vx-2}
\draw[blue,ultra thick] (\tone*\vx+\myxi*\wx,\tone*\vy+\myxi*\wy) -- (\ttwo*\vx+\myxi*\wx,\ttwo*\vy+\myxi*\wy) node [below,xshift=0.2cm] {$L_{\phi,-\overline s}$};

\pgfmathsetmacro{\s}{-4.3}
\pgfmathsetmacro{\n}{0}
\pgfmathsetmacro{\myxi}{\s-2*\n/\N}
\pgfmathsetmacro{\tone}{-\myxi*\wy/\vy+6}
\pgfmathsetmacro{\ttwo}{-\myxi*\wx/\vx-3}
\draw[violet,ultra thick] (\tone*\vx+\myxi*\wx,\tone*\vy+\myxi*\wy) -- (\ttwo*\vx+\myxi*\wx,\ttwo*\vy+\myxi*\wy) node [left,xshift=-0.2cm] {$L_{\phi, s_3}$};

\draw [dashed, magenta, ultra thick]  (0,0) -- (0,3.5) node[left ] {$L_{0,\overline s}$};

\end{tikzpicture}
\begin{tikzpicture}[scale=\figursize]
\draw(-\myheight,0) -- (0,0);
\draw(0,0) -- (0,\myheight);
\draw(-\myheight,0) -- (-\myheight,\myheight);
\draw(0,\myheight) -- (-\myheight,\myheight);

\draw[] (-\myheight,\myheight) node [right,yshift=-0.2cm]{b) case 1};

\clip (-\myheight,-\mydepth) rectangle (0+\mydepth,\myheight) ;

\pgfmathsetmacro{\myangle}{25}
\pgfmathsetmacro{\s}{-2}

\pgfmathsetmacro{\vx}{cos(\myangle)}
\pgfmathsetmacro{\vy}{sin(\myangle)}

\pgfmathsetmacro{\wx}{sin(\myangle)}
\pgfmathsetmacro{\wy}{-cos(\myangle)}

\pgfmathsetmacro{\tone}{(-3-\s*\wx)/\vx}
\pgfmathsetmacro{\ttwo}{-\s*\wx/\vx}

\begin{scope}
\clip (0,0) rectangle (-\myheight,\myheight);

\draw[brown, ultra thick] (\tone*\vx+\s*\wx,\tone*\vy+\s*\wy) -- (\ttwo*\vx+\s*\wx,\ttwo*\vy+\s*\wy) node[midway,above] {$r$};

\pgfmathsetmacro{\radius}{1}
\draw [blue] (\tone*\vx+\s*\wx,\tone*\vy+\s*\wy) node[xshift=0.8cm,yshift=.2cm] {$\phi^c$};
\draw [blue] (\tone*\vx+\s*\wx+\radius,\tone*\vy+\s*\wy) arc (0:\myangle:\radius);

\pgfmathsetmacro{\radius}{0.5}
\draw [blue] (\ttwo*\vx+\s*\wx,\ttwo*\vy+\s*\wy) node[xshift=-0.2cm,yshift=-.3cm] {$\phi$};
\draw [blue] (\ttwo*\vx+\s*\wx-\radius*\vx,\ttwo*\vy+\s*\wy-\radius*\vy) arc (-180:\myangle:\radius);

\draw [blue] (0,0.8+0.5) arc (90:180:0.5);
\draw [blue,fill] (-0.2,1) circle (0.03cm);

\end{scope}
\draw[olive, ultra thick] (\tone*\vx+\s*\wx,\tone*\vy+\s*\wy) -- (0,\tone*\vy+\s*\wy) node[midway,below] {$\delta_x$};
\draw[cyan,ultra thick] (\ttwo*\vx+\s*\wx,\ttwo*\vy+\s*\wy) -- (0,\tone*\vy+\s*\wy) node[midway,right] {};


\end{tikzpicture}
\begin{tikzpicture}[scale=\figursize]
\draw(-\myheight,0) -- (0,0);
\draw(0,0) -- (0,\myheight);
\draw(-\myheight,0) -- (-\myheight,\myheight);
\draw(0,\myheight) -- (-\myheight,\myheight);

\draw[] (-\myheight,\myheight) node [right,yshift=-0.2cm]{c) case 2};

\clip (-\myheight,-\mydepth) rectangle (0+\mydepth,\myheight) ;

\pgfmathsetmacro{\myangle}{25}

\pgfmathsetmacro{\scenter}{-(cos(\myangle)+sin(\myangle))*\myheight/2}
\pgfmathsetmacro{\t}{2}

\pgfmathsetmacro{\vx}{cos(\myangle)}
\pgfmathsetmacro{\vy}{sin(\myangle)}

\pgfmathsetmacro{\wx}{sin(\myangle)}
\pgfmathsetmacro{\wy}{-cos(\myangle)}

\pgfmathsetmacro{\tone}{-\scenter*\wy/\vy}
\pgfmathsetmacro{\ttwo}{-\scenter*\wx/\vx}

\pgfmathsetmacro{\tthree}{(-\myheight-\scenter*\wx)/\vx}

\draw[black, thick, dashed] (\tone*\vx+\scenter*\wx,\tone*\vy+\scenter*\wy) -- (\ttwo*\vx+\scenter*\wx,\ttwo*\vy+\scenter*\wy) node[midway,above,xshift=1cm,yshift=0.5cm] {$(0,0)$};

\pgfmathsetmacro{\radius}{1}
\draw [dashed](\tthree*\vx+\scenter*\wx,\tthree*\vy+\scenter*\wy) -- (\tthree*\vx+\scenter*\wx+\radius,\tthree*\vy+\scenter*\wy);

\draw [dashed](\tthree*\vx+\scenter*\wx,\tthree*\vy+\scenter*\wy) -- (\tthree*\vx+\scenter*\wx+\radius*\wx,\tthree*\vy+\scenter*\wy+\radius*\wy);

\draw [blue] (\tthree*\vx+\scenter*\wx,\tthree*\vy+\scenter*\wy) node[xshift=0.8cm,yshift=.2cm] {$\phi^c$};
\draw [blue] (\tthree*\vx+\scenter*\wx+\radius,\tthree*\vy+\scenter*\wy) arc (0:\myangle:\radius);

\draw [teal] (\tthree*\vx+\scenter*\wx+\radius*\wx,\tthree*\vy+\scenter*\wy+\radius*\wy) node[,yshift=.6cm] {$\phi$};
\draw [teal] (\tthree*\vx+\scenter*\wx+\radius*\wx,\tthree*\vy+\scenter*\wy+\radius*\wy) arc (270+\myangle:360:\radius);

\draw [fill] (-\myheight/2,\myheight/2) circle (0.05cm);

\pgfmathsetmacro{\s}{-0.9}
\pgfmathsetmacro{\t}{2}

\pgfmathsetmacro{\vx}{cos(\myangle)}
\pgfmathsetmacro{\vy}{sin(\myangle)}

\pgfmathsetmacro{\wx}{sin(\myangle)}
\pgfmathsetmacro{\wy}{-cos(\myangle)}

\pgfmathsetmacro{\tone}{-\s*\wy/\vy}
\pgfmathsetmacro{\ttwo}{-\s*\wx/\vx}

\draw[red, ultra thick] (\tone*\vx+\s*\wx,\tone*\vy+\s*\wy) -- (\ttwo*\vx+\s*\wx,\ttwo*\vy+\s*\wy) node[midway,above] {$r$};

\draw[olive,ultra thick] (\tone*\vx+\s*\wx,\tone*\vy+\s*\wy) -- (0,0) node[midway,below] {$a$};
\draw[cyan,ultra thick] (\ttwo*\vx+\s*\wx,\ttwo*\vy+\s*\wy) -- (0,0) node[midway,right] {$b$};

\pgfmathsetmacro{\radius}{1}


\draw [blue] (\tone*\vx+\s*\wx,0) node[xshift=0.8cm,yshift=.2cm] {$\phi^c$};
\draw [blue] (\tone*\vx+\s*\wx+\radius,0) arc (0:\myangle:\radius);


\draw[orange,ultra thick] (0,0) -- (\s*\wx,\s*\wy) node[midway,left]{$h$};
\draw[thick, magenta] (\scenter*\wx,\scenter*\wy) -- (\s*\wx,\s*\wy) node[midway,left]{$|s|$};

\draw [decorate,decoration={brace,amplitude=5pt,mirror},teal, very thick]
(0,0) -- (\scenter*\wx,\scenter*\wy)  node[midway,above,xshift=.2cm]{$\overline s$};

\end{tikzpicture}
\caption{Illustration supporting the proof of Lemma \ref{Lemma_weight_as_intersection_length}. In a), we depict the different cases passing through a square $Z$ for fixed $\phi$ (here $105^\circ$), where the teal lines describe the case 4 $|s|=\underline s(\phi)$ in which the corners are precisely hit, the brown line $|s_1|< \underline s(\phi)$ (case 1) and the red line $|s_2| \in [\underline s(\phi),\overline s(\phi)[$ (case 2). 
Moreover, the blue line is representative of case 5 (hitting exactly one corner), while the violet line representing case 3 does not hit $Z$, and the dashed magenta line is representative of case 6 with the line intersecting with one side of $Z$.
Figures b) and c)  detail  the geometry of case 1 and case 2, depicting relevant right triangles.}

\label{Fig_intersection_lengths}

\end{figure}

%
%
%
 \bibliographystyle{splncs04}
\bibliography{splncs04}

\end{document}